\RequirePackage{ifthen}
\RequirePackage{ifpdf}
\newcommand{\driverOption}{}
\ifthenelse{\boolean{pdf}}{
  \renewcommand{\driverOption}{pdftex}
} { 
  \renewcommand{\driverOption}{dvips}
}

\documentclass[reqno, 11pt, letterpaper, commented, oneside, \driverOption]{amsart}

\newboolean{isCommented}
\usepackage{geometry}
\usepackage{bbm}
\usepackage[final]{graphicx}
\usepackage{wrapfig}
\usepackage{upref}
\usepackage{enumitem}
\usepackage{latexsym}
\usepackage{amssymb}
\usepackage[ansinew]{inputenc}
\usepackage[T1]{fontenc}
\usepackage{lmodern,microtype}
\usepackage{mathtools}
\usepackage[mathcal]{euscript}
\usepackage[ruled,vlined,norelsize]{algorithm2e}
\usepackage{fancyvrb}
\usepackage{caption}
\usepackage[b]{esvect} 
\usepackage{mdframed}
\usepackage{xpatch}
\usepackage{cancel}
\usepackage{mycommands}
\usepackage{todonotes}

\newcommand{\hyperrefDriverOption}{}
\ifthenelse{\boolean{pdf}}{
	\renewcommand{\hyperrefDriverOption}{pdftex}
} { 
	\renewcommand{\hyperrefDriverOption}{hypertex}
}
\usepackage[\hyperrefDriverOption,
  colorlinks = false,
  pdftitle={On the central levels problem}]
  {hyperref}
\hypersetup{pdfauthor={Petr Gregor, Ondrej Micka, Torsten M\374tze}}

\usepackage{bookmark} 

\usepackage{tikz}
\usetikzlibrary{decorations.markings}
\usetikzlibrary{calc}
\usetikzlibrary{matrix}
\usetikzlibrary{backgrounds}
\tikzset{node_white/.style={circle, draw, fill=white, inner sep=0pt, text width=0pt, text height=0pt, text depth=0pt, minimum size = 4pt}}
\tikzset{node_black/.style={circle, draw, fill=black, inner sep=0pt, text width=0pt, text height=0pt, text depth=0pt, minimum size = 4pt}}
\tikzset{marker_red/.style={rectangle, draw, fill=red, inner sep=0pt, text width=0pt, text height=0pt, text depth=0pt, minimum size = 7pt}}
\tikzset{marker_blue/.style={diamond, draw, fill=blue, inner sep=0pt, text width=0pt, text height=0pt, text depth=0pt, minimum size = 9pt}}
\tikzset{node_huge/.style={circle, draw, fill=white, inner sep=0pt, text width=0pt, text height=0pt, text depth=0pt, minimum size = 33pt}}

\tikzset{line_solid/.style={draw, thick}}
\tikzset{line_dashed/.style={draw, thick, dashed}}
\tikzset{line_dotted/.style={draw, thick, dotted}}
\tikzset{line_gray_bg/.style={draw=black!40!white, line width=7pt, line cap=round, line join=round}}
\tikzset{dashdot/.style={dash pattern=on .8pt off 2pt on 4pt off 2pt}}
\tikzset{dotdot/.style={dash pattern=on 1pt off 1pt}}
\tikzset{ptr/.style={decoration={markings,mark=at position 1 with {\arrow[scale=1.5,>=latex]{>}}},postaction={decorate}}}

\ifthenelse{\boolean{isCommented}} {
	\newcommand{\PG}[1]{\marginpar{\parbox{4cm}{{\small {\bf PG:} #1}}}}
	\newcommand{\TM}[1]{\marginpar{\parbox{4cm}{{\small {\bf TM:} #1}}}}
} { 
	\newcommand{\PG}[1]{}
	\newcommand{\TM}[1]{}	
}

\newtheorem{theorem}{Theorem}
\newtheorem{lemma}[theorem]{Lemma}

\newtheorem{corollary}[theorem]{Corollary}
\newtheorem{conjecture}[theorem]{Conjecture}

\theoremstyle{definition}

\theoremstyle{remark}

\ifthenelse{\boolean{isCommented}} {
	\geometry{
	  hmargin={25mm, 50mm},
	  marginparwidth=40mm,
	  vmargin={25mm, 25mm},
	  headsep=10mm,
	  headheight=5mm,
	  footskip=10mm
	}
} { 
	\geometry{
	  hmargin={25mm, 25mm},
	  vmargin={25mm, 25mm},
	  headsep=10mm,
	  headheight=5mm,
	  footskip=10mm
	}
}

\makeatletter
\g@addto@macro{\endabstract}{\@setabstract}
\newcommand{\authorfootnotes}{\renewcommand\thefootnote{\@fnsymbol\c@footnote}}%
\makeatother

\begin{document}

\begin{center}

\authorfootnotes
\LARGE On the central levels problem\footnote{An extended abstract of this paper appeared in the Proceedings of the 47th International Colloquium on Automata, Languages, and Programming (ICALP 2020)~\cite{DBLP:conf/icalp/GregorMM20}.
This work was supported by Czech Science Foundation grant GA19-08554S. Torsten M\"utze is also supported by German Science Foundation grant~413902284.}
\vspace{2mm}

\large
Petr~Gregor\footnote{E-Mail: \texttt{gregor@ktiml.mff.cuni.cz}.}\textsuperscript{1},
Ond\v{r}ej Mi\v{c}ka\footnote{E-Mail: \texttt{micka@ktiml.mff.cuni.cz}.}\textsuperscript{1},
Torsten~M\"utze\footnote{E-Mail: \texttt{torsten.mutze@warwick.ac.uk}.}\textsuperscript{1,2}
\setcounter{footnote}{0}
\bigskip

\small

\textsuperscript{1}Department of Theoretical Computer Science and Mathematical Logic, \\ Charles University, Prague, Czech Republic \par
\textsuperscript{2}Department of Computer Science, University of Warwick, United Kingdom \par

\bigskip

\begin{minipage}{0.8\linewidth}
\textsc{Abstract.}
The \emph{central levels problem} asserts that the subgraph of the $(2m+1)$-dimensional hypercube induced by all bitstrings with at least $m+1-\ell$ many 1s and at most $m+\ell$ many 1s, i.e., the vertices in the middle $2\ell$ levels, has a Hamilton cycle for any $m\geq 1$ and $1\le \ell\le m+1$.
This problem was raised independently by Buck and Wiedemann, Savage, Gregor and {\v{S}}krekovski, and by Shen and Williams, and it is a common generalization of the well-known \emph{middle levels problem}, namely the case $\ell=1$, and classical binary Gray codes, namely the case $\ell=m+1$.
In this paper we present a general constructive solution of the central levels problem.
Our results also imply the existence of optimal cycles through any sequence of $\ell$ consecutive levels in the $n$-dimensional hypercube for any $n\ge 1$ and $1\le \ell \le n+1$.
Moreover, extending an earlier construction by Streib and Trotter, we construct a Hamilton cycle through the $n$-dimensional hypercube, $n\geq 2$, that contains the symmetric chain decomposition constructed by Greene and Kleitman in the 1970s, and we provide a loopless algorithm for computing the corresponding Gray code.
\end{minipage}

\vspace{2mm}

\begin{minipage}{0.8\linewidth}
\textsc{Keywords:} Gray code, Hamilton cycle, hypercube, middle levels, symmetric chain decomposition
\end{minipage}

\vspace{2mm}

\end{center}

\vspace{2mm}

\section{Introduction}

The \emph{$n$-dimensional hypercube}, or \emph{$n$-cube} for short, is the graph~$Q_n$ formed by all $\{0,1\}$-strings of length~$n$, with an edge between any two bitstrings that differ in exactly one bit.
This family of graphs has numerous applications in computer science and discrete mathematics, many of which are tied to famous problems and conjectures, such as the sensitivity conjecture of Nisan and Szegedy~\cite{MR1313531}, recently proved by Huang~\cite{MR4024566}; Erd\H{o}s and Guys' crossing number problem~\cite{MR382006} (see \cite{MR2440735}); F\"uredi's conjecture~\cite{Furedi1985} on equal-size chain partitions (see \cite{MR3349520}); Shearer and Kleitman's conjecture~\cite{MR532807} on orthogonal symmetric chain decompositions (see \cite{MR3952674}); the Ruskey-Savage problem~\cite{MR1201997} on matching extendability (see \cite{MR2354719,MR3936192}), and the conjectures of Norine, and Feder and Subi on edge-antipodal colorings~\cite{norine_2008,MR3043094}, to name just a few.

The focus of this paper are Hamilton cycles in the $n$-cube and its subgraphs.
A Hamilton cycle in a graph is a cycle that visits every vertex exactly once, and in the context of the $n$-cube, such a cycle is often referred to as a \emph{Gray code}.
Gray codes have found applications in signal processing, circuit testing, hashing, data compression, experimental design, binary counters, image processing, and in solving puzzles like the Towers of Hanoi or the Chinese rings; see Savage's survey~\cite{MR1491049}.
Gray codes are also fundamental for efficient algorithms to exhaustively generate combinatorial objects, a topic that is covered in depth in the most recent volume of Knuth's seminal series \emph{`The Art of Computer Programming'}~\cite{MR3444818}.

To start with, it is an easy exercise to show that the $n$-cube has a Hamilton cycle for any $n\geq 2$.
One such cycle is given by the classical \emph{binary reflected Gray code~$\Gamma_n$}~\cite{gray:patent}, defined inductively by $\Gamma_1:=0,1$ and $\Gamma_{n+1}:=0\Gamma_n,1\Gamma^R_n$, where $\Gamma^R$ denotes the reversal of the sequence~$\Gamma$, and $0\Gamma$ or $1\Gamma$ means prefixing all strings in the sequence~$\Gamma$ by~0 or~1, respectively.
For instance, this construction gives $\Gamma_2=00,01,11,10$ and $\Gamma_3=000,001,011,010,110,111,101,100$.
The problem of finding a Hamilton cycle becomes considerably harder when we restrict our attention to subgraphs of the cube induced by a sequence of consecutive levels, where the \emph{$k$th level} of $Q_n$, $0\le k\le n$, is the set of all bitstrings with exactly~$k$ many 1s in them.
One such instance is the famous \emph{middle levels problem}, raised in the 1980s by Havel~\cite{MR737021} and independently by Buck and Wiedemann~\cite{MR737262}, which asks for a Hamilton cycle in the subgraph of the $(2m+1)$-cube induced by levels~$m$ and~$m+1$.
This problem received considerable attention in the literature, and a construction of such a cycle for all~$m\geq 1$ was provided only recently by M\"utze~\cite{MR3483129}.
A much simpler construction was described subsequently by Gregor, M\"utze, and Nummenpalo~\cite{gregor-muetze-nummenpalo:18}.

There are several intuitive explanations why the middle levels problem is considerably harder than finding a Hamilton cycle in the entire cube.
First of all, the entire $(n+1)$-cube has a simple inductive structure: It consists of two copies of~$Q_n$ plus a perfect matching connecting the two copies, and hence a Hamilton cycle in~$Q_{n+1}$ can be obtained by gluing together two cycles in these copies of~$Q_n$ via two matching edges, which can yield the cycle~$\Gamma_{n+1}$ from before.
The subgraph of the $(2m+1)$-cube induced by the middle two levels, on the other hand, does not allow for such an inductive decomposition.
Note also that as a consequence of the inductive construction of the binary reflected Gray code~$\Gamma_n$, the number of times that each bit is flipped follows a highly skewed distribution~$(2,2,4,8,\ldots,2^{n-1})$.
In stark contrast to this, along \emph{any} Hamilton cycle in the middle levels graph, the number of times that each bit is flipped \emph{must} follow the uniform distribution~$(2C_m,2C_m,\ldots,2C_m)$ with $C_m=\frac{1}{2m+1}\binom{2m+1}{m}=\frac{1}{m+1}\binom{2m}{m}$ the $m$th Catalan number, an observation that makes an easy induction proof unlikely.

\subsection{Our results}

In this paper we consider the \emph{central levels problem}, a broad generalization of the middle levels problem:
Does the subgraph of the $(2m+1)$-cube induced by the middle $2\ell$ levels, i.e., by levels $m+1-\ell,\ldots,m+\ell$, have a Hamilton cycle for any $m\ge 1$ and $1\le \ell\le m+1$?
This problem was raised already in Buck and Wiedemann's paper~\cite{MR737262}, and was reiterated independently by Savage~\cite{MR1275228}, Gregor and {\v{S}}krekovski~\cite{MR2609124}, and by Shen and Williams~\cite{shen_williams_2019}.
Clearly, the case $\ell=1$ of the central levels problem is the aforementioned middle levels problem (solved in~\cite{MR3483129}).
Moreover, the case $\ell=2$ was solved affirmatively in a paper by Gregor, J\"ager, M\"utze, Sawada, and Wille~\cite{jaeger-et-al-journal:21}.
Also, the case $\ell=m+1$ is established by the binary reflected Gray code~$\Gamma_{2m+1}$.
Furthermore, the case $\ell=m$ was solved by Buck and Wiedemann~\cite{MR737262} and independently by El-Hashash and Hassan~\cite{MR1887372}, and in a more general setting by Locke and Stong~\cite{locke-stong:03}, and the case $\ell=m-1$ was settled in~\cite{MR2609124}.

The main contribution of this paper is to solve the central levels problem affirmatively in full generality; see Figure~\ref{fig:hc7}~(a)--(d).

\begin{theorem}
\label{thm:gmlc}
For any~$m\geq 1$ and $1\le \ell\le m+1$, the subgraph of the $(2m+1)$-cube induced by the middle $2\ell$~levels has a Hamilton cycle.
\end{theorem}

As the case $\ell=1$ of Theorem~\ref{thm:gmlc} has been proved before, the proof of Theorem~\ref{thm:gmlc} presented in this paper assumes that~$\ell\geq 2$.
Nevertheless, this proof can be seen as a generalization of the earlier proofs~\cite{MR3483129,gregor-muetze-nummenpalo:18} and~\cite{jaeger-et-al-journal:21} for the cases $\ell=1$ and $\ell=2$, respectively; see the remarks in Section~\ref{sec:comparison} below.

The most general question in this context is to ask for a Hamilton cycle in $Q_n$ that visits all vertices in \emph{any} sequence of $\ell$ consecutive levels, i.e., the levels need not be symmetric around the middle, and the dimension~$n$ needs not be odd.
Note however, that the $n$-cube is bipartite, and the partition classes are given by all even and odd levels, respectively.
Consequently, any subgraph consisting of a sequence of $\ell$ consecutive levels is also bipartite, and to circumvent the imbalances that prevent the existence of a Hamilton cycle for general~$n$ and~$\ell$, we have to slightly generalize the notion of Hamilton cycles, and there are two reasonable such generalized notions.
Firstly, a \emph{saturating cycle} in a bipartite graph is a cycle that visits all vertices in the smaller partition class.
If one partition class is empty, then the empty set is considered a saturating cycle, and if the smaller partition class has size~1, then a single edge is considered a saturating cycle.
Secondly, a \emph{tight enumeration} in a (bipartite) subgraph of the cube is a cyclic listing of all its vertices where the total number of bits flipped is exactly the number of vertices plus the difference in size between the two partition classes.
If the graph has only a single vertex, this vertex is considered a tight enumeration.
Clearly, if both partition classes have the same size, like in the central levels problem, then a saturating cycle and a tight enumeration are equal to a Hamilton cycle.
In fact, all cases of this more general problem on saturating cycles and tight enumerations, except the cases of the central levels problem, were solved affirmatively already in~\cite{MR3758308}, some of them conditional on a `yes' answer to the central levels problem.
Combining Theorem~\ref{thm:gmlc} with these previous results, we now also obtain an unconditional result for this more general question.

\begin{corollary}
\label{cor:sat-tight}
For any~$n\geq 1$ and $1\le \ell\le n+1$, the subgraph of the $n$-cube induced by any sequence of~$\ell$ consecutive levels has both a saturating cycle and a tight enumeration.
\end{corollary}

\begin{proof}[Proof of Corollary~\ref{cor:sat-tight}]
If $\ell=1$, then all vertices are in the same partition class, and the other partition class is empty.
It follows that the statement is trivially true for $\ell=1$ and saturating cycles.
To prove it for $\ell=1$ and tight enumerations, first note that level~$0$ or level~$n$ consist only of a single vertex, which is trivially a tight enumeration.
Otherwise we use a well-known result of Tang and Liu~\cite{MR0349274}, who showed that for any $n\geq 2$ and $1\leq k\leq n-1$, there is a cyclic listing of all bitstrings of~$Q_n$ on level~$k$ such that any two consecutive strings differ in a transposition of~0 and~1.
As two bits are flipped in each step, this is a tight enumeration.
We now consider the case $\ell=2$ for saturating cycles:
If one of the two levels is level~$0$ or level~$n$, which contains only a single vertex, then a single edge is a trivial saturating cycle.
Otherwise the result follows from \cite[Theorem~9]{MR3759914}.

All remaining cases for saturating cycles and tight enumerations are covered by Theorems~5 and~6 proved in~\cite{MR3758308}, respectively.
Specifically, part~(iv) in both of these theorems is conditional on the validity of Theorem~\ref{thm:gmlc}, which we establish in this paper.
\end{proof}

The central levels problem studied in this paper is also closely related to another famous problem, which asks about Hamilton cycles in so-called bipartite Kneser graphs.
The \emph{bipartite Kneser graph~$H_{n,k}$}, defined for any integers $k\ge 0$ and $n\ge 2k+1$, is the bipartite graph whose vertex partition is given by all vertices on level~$k$ and~$n-k$ of~$Q_n$, with an edge between any two bitstrings that differ in exactly $n-2k$ bits.
That is, the edges in~$H_{n,k}$ correspond to a level-monotone path in~$Q_n$ between a vertex on level~$k$ and a vertex on level~$n-k$.
It was shown that $H_{n,k}$ has a Hamilton cycle for all~$k\geq 1$ and $n\geq 2k+1$ in~\cite{MR3759914}, completing a long line of previous partial results~\cite{MR1152123,MR1271867,MR1778200,MR1999733}.

\begin{figure}
\includegraphics{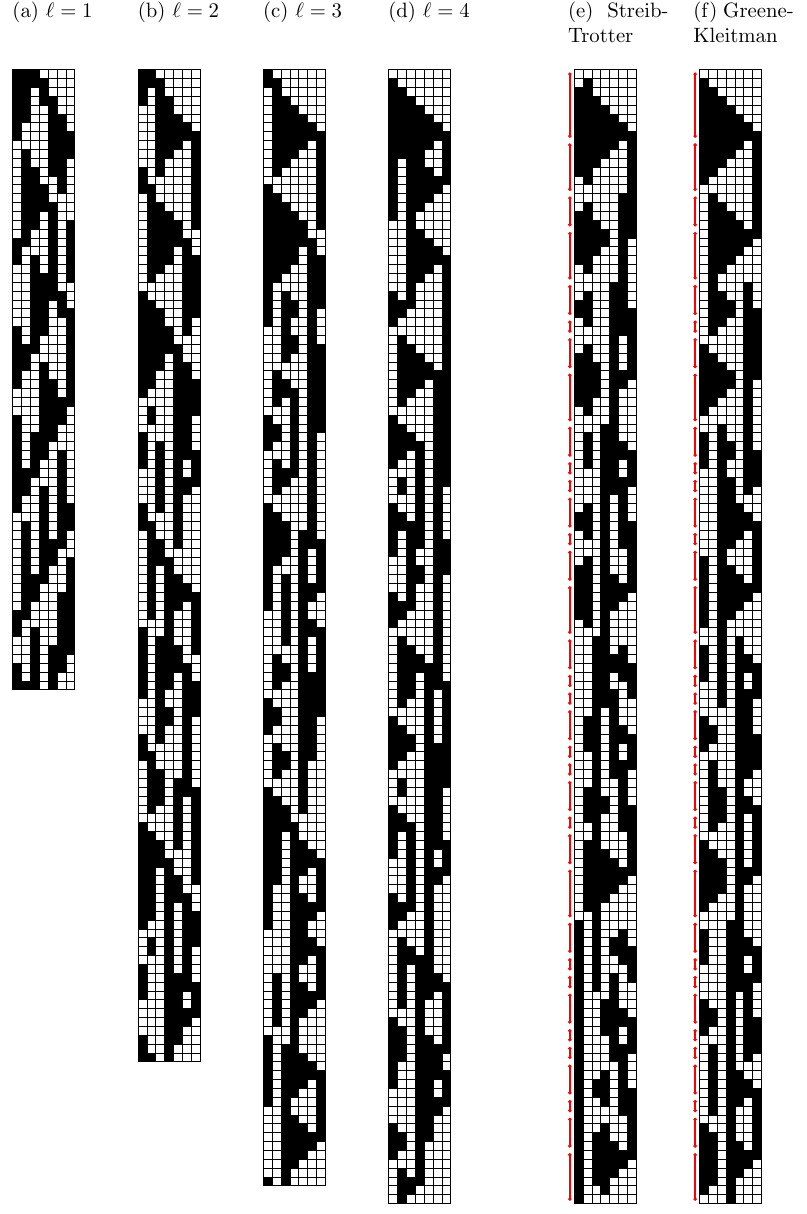}
\caption{(a)-(d)~The Hamilton cycles in $Q_{7,\ell}$ for $\ell=1,2,3,4$ constructed as in our proof of Theorem~\ref{thm:gmlc}.
(e)~The Hamilton cycle in $Q_7$ containing an SCD obtained from the Streib-Trotter construction, with symmetric chains highlighted on the side.
(f)~The Hamilton cycle in $Q_7$ containing the Greene-Kleitman SCD obtained from our proof of Theorem~\ref{thm:gk}.
In this figure, 1-bits are drawn as black squares, 0-bits as white squares.
}
\label{fig:hc7}
\end{figure}

An essential tool in our proof of Theorem~\ref{thm:gmlc} are symmetric chain decompositions.
This is a well-known concept from the theory of posets, which we now define specifically for the $n$-cube using graph-theoretic language.
A \emph{symmetric chain} in~$Q_n$ is a path $(x_k,x_{k+1},\ldots,x_{n-k})$ in the $n$-cube where $x_i$ is from level~$i$ for all $k\le i\le n-k$, and a \emph{symmetric chain decomposition}, or SCD for short, is a partition of the vertices of~$Q_n$ into symmetric chains.
It is well-known that the $n$-cube has an SCD for all $n\geq 1$, and the simplest explicit construction was given by Greene and Kleitman~\cite{MR0389608} (see Section~\ref{sec:scd} below).
Streib and Trotter~\cite{MR3268651} first investigated the interplay between SCDs and Hamilton cycles in the $n$-cube, and they described an SCD in~$Q_n$ that can be extended to a Hamilton cycle; see Figure~\ref{fig:hc7}~(e).
Streib and Trotter's SCD, however, is different from the aforementioned Greene-Kleitman SCD.
In this paper, we extend Streib and Trotter's result as follows; see Figure~\ref{fig:hc7}~(f).

\begin{theorem}
\label{thm:gk}
For any~$n\geq 2$, the Greene-Kleitman SCD can be extended to a Hamilton cycle in~$Q_n$.
\end{theorem}

The Greene-Kleitman SCD has found a large number of applications in the literature, e.g., to construct rotation-symmetric Venn diagrams \cite{MR2034416,MR2268388}, to solve different variants of the Littlewood-Offord problem~\cite{MR184865,MR1183998} (see also~\cite[Chap.~4]{MR866142}), or to learn monotone Boolean functions \cite[Sec.~7.2.1.6]{MR3444818} (see also \cite{MR0043115,MR0319772,MR0450151,MR532807,MR1765729}).
Knowing that this SCD extends to a Hamilton cycle and that it is a crucial ingredient for solving the general central levels problem adds to this list of interesting properties and applications.
Observe also that a Hamilton cycle that extends an SCD has the intriguing property that it minimizes the number of changes of direction from moving up to moving down, or vice versa, between consecutive levels in the cube.
For comparison, the monotone paths constructed by Savage and Winkler~\cite{MR1329390} maximize these changes.

Motivated by these results and by the aforementioned conjecture of Ruskey and Savage~\cite{MR1201997} that every matching in $Q_n$ extends to a Hamilton cycle, we raise the following brave conjecture.

\begin{conjecture}
\label{con:ext}
Every SCD can be extended to a Hamilton cycle in~$Q_n$.
\end{conjecture}

Although every SCD of~$Q_n$ is the union of two matchings, there are matchings in~$Q_n$ that do not extend to an SCD; take for example the two edges obtained by starting at the vertices~$0^n$ and~$1^n$ and flipping the same bit.
Consequently, an affirmative answer to Conjecture~\ref{con:ext} would cover only some cases of the Ruskey-Savage conjecture.

\subsection{Efficient algorithms}

We now discuss the algorithmic problem of efficiently computing the Hamilton cycles constructed in the proofs of Theorems~\ref{thm:gmlc} and~\ref{thm:gk}.
The ultimate goal for any Gray code problem is a \emph{loopless} algorithm, a notion that was introduced by Ehrlich~\cite{MR0366085}.
The running time of such an algorithm is~$\cO(1)$ per generated vertex, and the initialization time is `reasonable', linear in~$n$, say.
Furthermore, the overall memory requirement should also be polynomial in~$n$, ideally linear, and this \emph{excludes} the space for storing the Hamilton cycle (cf.\ Ruskey's~\cite{ruskey_2003} `Don't count the output principle'), which may not be needed for some applications.

For the Hamilton cycle in Theorem~\ref{thm:gk}, we present such a loopless algorithm.
Specifically, the algorithm uses $\cO(1)$ time in every iteration to compute the next bitstring along the Hamilton cycle, its initialization time is~$\cO(n)$, and its space requirement is~$\cO(n)$.
In this paper, we provide a pseudocode description of this algorithm, and we implemented it in~C++, available for download and for demonstration on the Combinatorial Object Server~\cite{cos_chains}.

On the other hand, for the Hamilton cycles constructed in the proof of Theorem~\ref{thm:gmlc}, we have no efficient generation algorithm, and it remains a challenging open problem to find one.
While our proof of Theorem~\ref{thm:gmlc} is constructive and translates straightforwardly into an algorithm that computes the desired Hamilton cycle in time and space that are polynomial in the size of the graph (the middle $2\ell$ levels of $Q_n$, $n:=2m+1$), these quantities are clearly exponential in~$n$.
There are fundamental obstacles that prevent us from obtaining algorithms with polynomial bounds from our general proof of the central levels problem, as explained below (at the end of Section~\ref{sec:connect}).
The only two special cases of the central levels problem for which loopless algorithms are known, which were found prior to our work, are the binary reflected Gray code~$\Gamma_n$~\cite{MR0424386} and the middle levels problem~\cite{MR3627876}, i.e., the extreme cases $\ell=m+1$ and $\ell=1$, respectively.

\subsection{Proof ideas}

We first describe the ideas for proving Theorem~\ref{thm:gmlc}.
For any $m\geq 1$ we define $n:=2m+1$, and for $1\le \ell\le m+1$ we let $Q_{n,\ell}$ denote the subgraph of~$Q_n$ induced by the middle $2\ell$ levels.
To prove that $Q_{n,\ell}$ has a Hamilton cycle for general~$m$ and~$\ell$, we combine and generalize the tools and techniques developed for the cases $\ell=1$ and $\ell=2$ in~\cite{gregor-muetze-nummenpalo:18} and~\cite{jaeger-et-al-journal:21}, respectively.
Our proof proceeds in two steps: In a first step, we construct a \emph{cycle factor} in $Q_{n,\ell}$, i.e., a collection of disjoint cycles which together visit all vertices of $Q_{n,\ell}$.
In a second step, we use local modifications to join the cycles in the factor to a single Hamilton cycle.
Essentially, this technique reduces the Hamiltonicity problem in~$Q_{n,\ell}$ to proving that a suitably defined auxiliary graph is connected, which is much easier.

In fact, the predecessor paper~\cite{jaeger-et-al-journal:21} already proved the existence of a cycle factor in~$Q_{n,\ell}$, but this construction does not seem to yield a factor that would be amenable to analysis.
In this paper, we therefore construct another cycle factor in~$Q_{n,\ell}$, based on modifying the aforementioned Greene-Kleitman SCD of~$Q_n$ by the lexical matchings introduced by Kierstead and Trotter~\cite{MR962224}.
The resulting cycle factor in~$Q_{n,\ell}$ has a rich structure, in particular the number of cycles and their lengths can be described combinatorially.

The simplest way to join two cycles~$C$ and~$C'$ from this factor to a single cycle is to consider a 4-cycle~$F$ that shares exactly one edge with each of the cycles~$C$ and~$C'$ (the other two edges of~$F$ must then go between~$C$ and~$C'$), and to take the symmetric difference of the edge sets of~$C\cup C'$ and of~$F$, yielding a single cycle $(C\cup C')\bigtriangleup F$ on the same vertex set as~$C\cup C'$.
We refer to such a cycle~$F$ as a \emph{flipping $4$-cycle}.
For example, if we interpret the binary reflected Gray code~$\Gamma_n$ as a cycle in~$Q_n$, we see that $\Gamma_{n+1}=(0\Gamma_n \cup 1\Gamma_n^R)\bigtriangleup F$ where $F$ is the 4-cycle $F=0^{n+1},010^{n-1},110^{n-1},10^n$.
In addition to flipping 4-cycles, we also use flipping 6-cycles, which intersect with the two cycles to be joined in a slightly more complicated way, albeit with the same effect of joining them to a single cycle.
The most technical aspect of this part of the proof is to ensure that all flipping cycles used are edge-disjoint, so that the joining operations do not interfere with each other.

To prove Theorem~\ref{thm:gk}, we proceed by induction from dimension~$n$ to~$n+2$, treating the cases of even and odd~$n$ separately.
We first specify a particular ordering of all chains of the Greene-Kleitman SCD, and then show that this ordering admits a matching that alternatingly joins the bottom or top vertices of any two consecutive chains in our ordering.
In fact, there is a close relation between our proofs of Theorem~\ref{thm:gmlc} and~\ref{thm:gk}:
The aforementioned construction of a cycle factor in~$Q_{n,\ell}$ is particularly nice for $\ell=m+1$, i.e., for the case where we consider the entire cube.
Specifically, in this case our cycle factor contains \emph{all} chains from the Greene-Kleitman SCD.
These cycles can be joined to a single Hamilton cycle in such a way, so as to give exactly the aforementioned Hamilton cycle constructed for proving Theorem~\ref{thm:gk}.

\subsection{Outline of this paper}

In Section~\ref{sec:prel} we build up the necessary preliminaries.
Specifically, we introduce various Catalan bijections, the Greene-Kleitman SCD, and lexical matchings.
In Section~\ref{sec:cfac} we describe our construction of a cycle factor in $Q_{n,\ell}$, and we analyze its structure in Section~\ref{sec:struc}, identifying two types of cycles, called short and long cycles.
In Section~\ref{sec:flip4} we describe the flipping $4$-cycles that we use for attaching the short cycles to the long cycles.
In Section~\ref{sec:flip6} we describe the flipping $6$-cycles that we use for joining the long cycles to a Hamilton cycle in~$Q_{n,\ell}$.
In Section~\ref{sec:connect} we put together all ingredients for proving Theorem~\ref{thm:gmlc}, reducing the Hamiltonicity problem to a connectivity problem in a suitably defined auxiliary graph.
In Section~\ref{sec:gk} we present our proof of Theorem~\ref{thm:gk} and the corresponding loopless algorithm.
This part can be read independently of the previous parts, though it is closely related.

\section{Preliminaries}
\label{sec:prel}

We begin by introducing some terminology that is used throughout the following sections.

\subsection{Bitstrings, lattice paths, and rooted trees}
\label{sec:dyck}

For any string~$x$ and any integer~$k\geq 0$, we let $x^k$ denote the concatenation of~$k$ copies of~$x$.
We often interpret a bitstring~$x$ as a path in the integer lattice~$\mathbb{Z}^2$ starting at the origin~$(0,0)$, where every 0-bit is interpreted as a $\diagdown$-step that changes the current coordinate by~$(+1,-1)$ and every 1-bit is interpreted as an $\diagup$-step that changes the current coordinate by~$(+1,+1)$; see Figure~\ref{fig:tree}.

\begin{figure}[h!]
\begin{tikzpicture}[scale=0.5]

\draw[ptr] (-0.5,0) -- (15,0);
\draw[ptr] (0,0.5) -- (0,-4);

\node[node_black] (0) at (0,0) {};
\node[node_black] (1) at (1,-1) {};
\node[node_black] (2) at (2,-2) {};
\node[node_black] (3) at (3,-3) {};
\node[node_black] (4) at (4,-2) {};
\node[node_black] (5) at (5,-1) {};
\node[node_black] (6) at (6,-2) {};
\node[node_black] (7) at (7,-1) {};
\node[node_black] (8) at (8,-0) {};
\node[node_black] (9) at (9,-1) {};
\node[node_black] (10) at (10,0) {};
\node[node_black] (11) at (11,-1) {};
\node[node_black] (12) at (12,-2) {};
\node[node_black] (13) at (13,-1) {};
\node[node_black] (14) at (14,0) {};

\path[line_solid] (0) -- (1) -- (2) -- (3) -- (4) -- (5) -- (6) -- (7) -- (8) -- (9) -- (10) -- (11) -- (12) -- (13) -- (14);

\node[node_black] (Z) at (15.5,-3.5) {};
\node[node_black] (A) at (15.5,-2.5) {} edge[thick] (Z);
\node[node_black] (B) at (17.5,-2.5) {};
\node[node_black] (C) at (16.5,-1.5) {} edge[thick] (A) edge[thick] (B);
\node[node_black] (D) at (19.5,-1.5) {};
\node[node_black] (E) at (22.5,-2.5) {};
\node[node_black] (F) at (22.5,-1.5) {} edge[thick] (E);
\node[node_black] at (19.5,-0.5) {} edge[thick] (C) edge[thick] (D) edge[thick] (F);

\node at (-5.5,-2) {$x=00011011010011\in D_{14}$};

\end{tikzpicture} 
\caption{The correspondence between bitstrings (left), lattice paths (middle) and rooted trees (right).}
\label{fig:tree}
\end{figure}
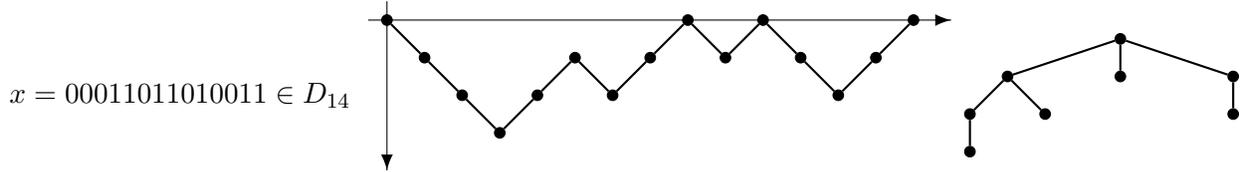

Let $D_{2k}$ denote the set of bitstrings with exactly $k$ many 1s and $k$ many 0s, such that in every prefix, the number of~0s is at least as large as the number of 1s.
We also define $D:=\bigcup_{k\geq 0} D_{2k}$.
Note that $D_0=\{\varepsilon\}$, where $\varepsilon$ denotes the empty bitstring.
In terms of lattice paths, $D$ corresponds to so-called \emph{Dyck paths} that never move above the line~$y=0$ and end on this line.
If a lattice path $x$ contains a substring $u\in D$, then we refer to this substring~$u$ as a \emph{valley} in~$x$.
Any nonempty bitstring~$x\in D$ can be written uniquely as $x=0\,u\,1\,v$ and as $x=u'\,0\,v'\,1$ with~$u,v,u',v'\in D$.
We refer to this as the \emph{left} and \emph{right factorization} of~$x$, respectively.
The set~$D'_{2k}$ is defined similarly as~$D_{2k}$, but we require that in exactly one prefix, the number of~0s is strictly smaller than the number of~1s.
That is, the lattice paths corresponding to the bitstrings in $D'_{2k}$ move above the line $y=0$ exactly once.

\begin{wrapfigure}{r}{0.45\textwidth}
\includegraphics{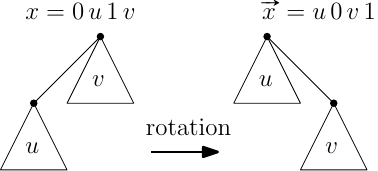}
\caption{Definition of tree rotation.}
\label{fig:rot}
\end{wrapfigure}
An \emph{(ordered) rooted tree} is a tree with a distinguished root vertex, and the children of each vertex have a specified left-to-right ordering.
We think of a rooted tree as a tree embedded in the plane with the root on top, with downward edges leading from any vertex to its children, and the children appear in the specified left-to-right ordering.
Using a standard Catalan bijection (see~\cite{MR3467982}), every Dyck path~$x\in D_{2k}$ can be interpreted as a rooted tree with $k$~edges; see Figure~\ref{fig:tree}.
Given a rooted tree~$x\in D$, we may rotate the tree to the right, yielding the tree~$\rot{x}$; see Figure~\ref{fig:rot}.
More precisely, $\rot{x}$ is obtained from~$x$ by taking the leftmost child of the root of~$x$ as the new root.
In terms of bitstrings, if $x$ has the left factorization $x=0\,u\,1\,v$ with $u,v\in D$, then $\rot{x}$ has the right factorization $\rot{x}=u\,0\,v\,1$.
Note that we have $\rot{D}=D$.

\subsection{The Greene-Kleitman SCD}
\label{sec:scd}

We now describe Greene and Kleitman's~\cite{MR0389608} construction of an SCD in the $n$-cube; see Figure~\ref{fig:paren}.
For any vertex~$x$ of the $n$-cube, we interpret the 0s in~$x$ as opening brackets and the 1s as closing brackets.
By matching closest pairs of opening and closing brackets in the natural way, the chain containing~$x$ is obtained by flipping the leftmost unmatched~0 to ascend the chain, or the rightmost unmatched~1 to descend the chain, until no more unmatched bits can be flipped.
It is easy to see that this indeed yields an SCD of the $n$-cube for any~$n\geq 1$.
In the rest of the paper, we always work with this SCD due to Greene and Kleitman, and whenever referring to a chain, we mean a chain from this decomposition.

\begin{figure}[h!]
\begin{tikzpicture}[every node/.style={gray, inner sep=1pt}]
 \pgfsetmatrixcolumnsep{0pt}
 \pgfsetmatrixrowsep{0pt}
 \matrix (m) [matrix of nodes] {
  1&1&1&1&1&0&1&1&0&1&1&0&1&0&0&1&1&1&1&|[black]|1&0&1 \\
  1&1&1&1&1&0&1&1&0&1&1&0&1&0&0&1&1&1&|[black]|1&|[black]|0&0&1 \\
  1&1&1&1&1&0&1&1&0&1&1&0&1&0&0&1&1&|[black]|1&|[black]|0&0&0&1 \\
  1&1&1&1&1&0&1&1&0&1&|[black]|1&0&1&0&0&1&1&|[black]|0&0&0&0&1 \\
  1&1&1&1&1&0&1&|[black]|1&0&1&|[black]|0&0&1&0&0&1&1&0&0&0&0&1 \\
  1&1&1&1&|[black]|1&0&1&|[black]|0&0&1&0&0&1&0&0&1&1&0&0&0&0&1 \\
  1&1&1&|[black]|1&|[black]|0&0&1&0&0&1&0&0&1&0&0&1&1&0&0&0&0&1 \\[5pt]
  |[black]|1&|[black]|1&|[black]|1&|[black]|0&|[black]|0&|[black]|0&|[black]|1&|[black]|0&|[black]|0&|[black]|1&|[black]|0&|[black]|0&|[black]|1&|[black]|0&|[black]|0&|[black]|1&|[black]|1&|[black]|0&|[black]|0&|[black]|0&|[black]|0&|[black]|1 \\[-1pt]
  |[black]|\footnotesize\texttt)&|[black]|\footnotesize\texttt)&|[black]|\footnotesize\texttt)&|[black]|\footnotesize\texttt(&|[black]|\footnotesize\texttt(&|[black]|\footnotesize\texttt(&|[black]|\footnotesize\texttt)&|[black]|\footnotesize\texttt(&|[black]|\footnotesize\texttt(&|[black]|\footnotesize\texttt)&|[black]|\footnotesize\texttt(&|[black]|\footnotesize\texttt(&|[black]|\footnotesize\texttt)&|[black]|\footnotesize\texttt(&|[black]|\footnotesize\texttt(&|[black]|\footnotesize\texttt)&|[black]|\footnotesize\texttt)&|[black]|\footnotesize\texttt(&|[black]|\footnotesize\texttt(&|[black]|\footnotesize\texttt(&|[black]|\footnotesize\texttt(&|[black]|\footnotesize\texttt) \\[13pt]
  1&|[black]|1&|[black]|0&0&0&0&1&0&0&1&0&0&1&0&0&1&1&0&0&0&0&1 \\
  |[black]|1&|[black]|0&0&0&0&0&1&0&0&1&0&0&1&0&0&1&1&0&0&0&0&1 \\
  |[black]|0&0&0&0&0&0&1&0&0&1&0&0&1&0&0&1&1&0&0&0&0&1 \\
 };
 \draw ($(m-9-6.south)+(0,0.1)$) -- ($(m-9-6.south)-(0,0.15)$) -- ($(m-9-7.south)-(0,0.15)$) -- ($(m-9-7.south)+(0,0.1)$);
 \draw ($(m-9-9.south)+(0,0.1)$) -- ($(m-9-9.south)-(0,0.15)$) -- ($(m-9-10.south)-(0,0.15)$) -- ($(m-9-10.south)+(0,0.1)$);
 \draw ($(m-9-12.south)+(0,0.1)$) -- ($(m-9-12.south)-(0,0.15)$) -- ($(m-9-13.south)-(0,0.15)$) -- ($(m-9-13.south)+(0,0.1)$);
 \draw ($(m-9-15.south)+(0,0.1)$) -- ($(m-9-15.south)-(0,0.15)$) -- ($(m-9-16.south)-(0,0.15)$) -- ($(m-9-16.south)+(0,0.1)$);
 \draw ($(m-9-14.south)+(0,0.1)$) -- ($(m-9-14.south)-(0,0.25)$) -- ($(m-9-17.south)-(0,0.25)$) -- ($(m-9-17.south)+(0,0.1)$); 
 \draw ($(m-9-21.south)+(0,0.1)$) -- ($(m-9-21.south)-(0,0.15)$) -- ($(m-9-22.south)-(0,0.15)$) -- ($(m-9-22.south)+(0,0.1)$);
 \draw (m-8-1.north west) rectangle (m-9-22.south east);
 
\begin{scope}[on background layer]
 \fill[red!50] (m-8-3.north west) rectangle (m-10-3.south east);
 \fill[red!50] (m-10-2.north west) rectangle (m-11-2.south east);
 \fill[red!50] (m-11-1.north west) rectangle (m-12-1.south east);
 \fill[red!50] (m-7-4.north west) rectangle (m-8-4.south east);
 \fill[red!50] (m-6-5.north west) rectangle (m-7-5.south east);
 \fill[red!50] (m-5-8.north west) rectangle (m-6-8.south east);
 \fill[red!50] (m-4-11.north west) rectangle (m-5-11.south east);
 \fill[red!50] (m-3-18.north west) rectangle (m-4-18.south east);
 \fill[red!50] (m-2-19.north west) rectangle (m-3-19.south east);
 \fill[red!50] (m-1-20.north west) rectangle (m-2-20.south east);
\end{scope}

\node[left,anchor=east,black,xshift=-2mm] at (m-8-1) {$x={}$};
\end{tikzpicture}
\caption{Construction of the Greene-Kleitman SCD containing a bitstring~$x$ via parenthesis matching.
The highlighted bits are the leftmost unmatched~0 and the rightmost unmatched~1 in each bitstring.
}
\label{fig:paren}
\end{figure}

Each chain $C$ of length~$h$ in~$Q_n$ can be encoded compactly as a string of length~$n$ over the alphabet~$\{0,1,*\}$ in the form
\begin{equation}
\label{eq:C}
C=u_0*u_1*\cdots *u_{h-1}*u_h,
\end{equation}
where $u_0,\ldots,u_h\in D$.
The symbols $*$ represent unmatched positions, and the vertices along the chain are obtained by replacing the $*$s by 1s followed by 0s in all $h+1$ possible ways; see~\eqref{eq:xi} below.
For example, the chain shown in Figure~\ref{fig:paren} is $C={*}{*}{*}{*}{*}01{*}01{*}010011{*}{*}{*}01$, so we have $u_0=u_1=u_2=u_3=u_4=u_8=u_9=\varepsilon$, $u_5=u_6=u_{10}=01$, and $u_7=010011$.

We distinguish four types of chains depending on whether~$u_0$ and~$u_h$, i.e., the first and last valleys in~\eqref{eq:C}, are empty or not.
These chain types are denoted by $[--]$, $[+-]$, $[++]$, and $[-+]$, where the first symbol is $-$ if $u_0=\varepsilon$ and $+$ otherwise, and the second symbol is $-$ if $u_h=\varepsilon$ and $+$ otherwise.
For example, the chain in Figure~\ref{fig:paren} is a $[-+]$-chain.
We also use the symbol~$?$ in these type specifications if we do not know whether a valley is empty or not.
Note that there is no $[--]$-chain in~$Q_n$ of length~$h=1$ unless $n=1$.

Given a chain~$C$ of length~$h$ as in~\eqref{eq:C}, the $i$th vertex of $C$ from the bottom is
\begin{equation}
\label{eq:xi}
x=u_0 \,1\,\cdots u_{i-1}\,1\,u_i\,0\,u_{i+1}\cdots \,0\,u_h
\end{equation}
where $i=0,\ldots,h$, and this vertex belongs to level $k=\frac{n-h}{2}+i$.
Note that every vertex~$x$ of~$Q_n$ can be written uniquely in the form~\eqref{eq:xi}, and we refer to this as the \emph{chain factorization of~$x$}.
For the following arguments, it will be crucial to consider the lattice path representation of~$x$, with the valleys~$u_0,\ldots,u_h$ that are separated by $i$ many $\diagup$-steps, followed by $h-i$ many $\diagdown$-steps, i.e., the valley~$u_i$ is the highest one on the lattice path.

We use $C_{h,i}$, $0\le i \le h$, to denote the set of the $i$th vertices in all chains of length~$h$.
Moreover, we partition $C_{h,i}$ into two sets~$C^-_{h,i}$ and~$C^+_{h,i}$, depending on whether the valley~$u_i$ in~\eqref{eq:xi} is empty or nonempty, respectively.
Clearly, $C^+_{h,h}$ are exactly the top vertices of $[?+]$-chains of length~$h$ and $C^+_{h,0}$ are exactly the bottom vertices of $[+?]$-chains of length~$h$, and similarly with~$-$ instead of~$+$.
Note that the sets $C_{h,i}$ are empty if $n$ is odd and $h$ is even, or vice versa.

\subsection{Lexical matchings}

\begin{figure}
\tikzset{dlabela/.style={xshift=1.4mm,yshift=2mm}}
\tikzset{dlabelb/.style={xshift=-1.4mm,yshift=-2mm}}
\tikzset{ulabela/.style={xshift=-1.4mm,yshift=2mm}}
\tikzset{ulabelb/.style={xshift=1.4mm,yshift=-2mm}}

\begin{tikzpicture}[scale=0.4]

\draw[ptr] (-1,0) -- (26,0);
\draw[ptr] (0,-5) -- (0,3.5);

\node[node_black] (0) at (0,0) {};
\node[node_black] (1) at (1,1) {};
\node[node_black] (2) at (2,2) {};
\node[node_black] (3) at (3,3) {};
\node[node_black] (4) at (4,2) {};
\node[node_black] (5) at (5,1) {};
\node[node_black] (6) at (6,0) {};
\node[node_black] (7) at (7,1) {};
\node[node_black] (8) at (8,0) {};
\node[node_black] (9) at (9,-1) {};
\node[node_black] (10) at (10,0) {};
\node[node_black] (11) at (11,-1) {};
\node[node_black] (12) at (12,-2) {};
\node[node_black] (13) at (13,-1) {};
\node[node_black] (14) at (14,-2) {};
\node[node_black] (15) at (15,-3) {};
\node[node_black] (16) at (16,-2) {};
\node[node_black] (17) at (17,-1) {};
\node[node_black] (18) at (18,-2) {};
\node[node_black] (19) at (19,-3) {};
\node[node_black] (20) at (20,-4) {};
\node[node_black] (21) at (21,-5) {};
\node[node_black] (22) at (22,-4) {};

\path[line_solid,gray] (0) -- (1) -- (2) -- (3) (6) -- (7) (9) -- (10) (12) -- (13) (15) -- (16) -- (17) (21) -- (22);
\draw[ultra thick] (3) -- node[dlabela] {0} (4) -- node[dlabela] {1} (5) -- node[dlabela] {3} (6) (7) -- node[dlabela] {2} (8) -- node[dlabelb] {5} (9) (10) -- node[dlabelb] {4} (11) -- node[dlabelb] {8} (12) (13) -- node[dlabela] {7} (14) -- node[dlabelb] {10} (15) (17) -- node[dlabela] {6} (18) -- node[dlabela] {9} (19) (20) -- node[dlabelb] {12} (21);
\draw[line width=2pt, red] (19) -- node[dlabela] {11} (20);

\node[anchor=west] at (0.5,-3.2) {$x=1110001001001001100001$};
\node[anchor=east] at (21.5,1) {$x=M_{n,k}^{11,\downarrow}(y)$};
\node at (24,3) {$x^\diagdown$};
\node at (-4,0) {level $k=9$};
\node at (23.8,-0.7) {$n=22$};

\draw[red,|-to] (18,2.5) -- node[right] {$M_{n,k}^{11,\uparrow}$} (18,4.5);
\draw[red,|-to] (15,4.5) -- node[left] {$M_{n,k}^{11,\downarrow}$} (15,2.5);

\draw[dashed] (22,-5) -- (22,12);

\begin{scope}[yshift=8.5cm]
\draw[ptr] (-1,0) -- (26,0);
\draw[ptr] (0,-4) -- (0,3.5);

\node[node_black] (0) at (0,0) {};
\node[node_black] (1) at (1,1) {};
\node[node_black] (2) at (2,2) {};
\node[node_black] (3) at (3,3) {};
\node[node_black] (4) at (4,2) {};
\node[node_black] (5) at (5,1) {};
\node[node_black] (6) at (6,0) {};
\node[node_black] (7) at (7,1) {};
\node[node_black] (8) at (8,0) {};
\node[node_black] (9) at (9,-1) {};
\node[node_black] (10) at (10,0) {};
\node[node_black] (11) at (11,-1) {};
\node[node_black] (12) at (12,-2) {};
\node[node_black] (13) at (13,-1) {};
\node[node_black] (14) at (14,-2) {};
\node[node_black] (15) at (15,-3) {};
\node[node_black] (16) at (16,-2) {};
\node[node_black] (17) at (17,-1) {};
\node[node_black] (18) at (18,-2) {};
\node[node_black] (19) at (19,-3) {};
\node[node_black] (20) at (20,-2) {};
\node[node_black] (21) at (21,-3) {};
\node[node_black] (22) at (22,-2) {};
\node[node_white] (23) at (23,-1) {};
\node[node_white] (24) at (24,0) {};
\node[node_white] (25) at (25,1) {};

\tikzset{every node/.style={inner sep=1pt}}

\path[line_solid,gray] (3) -- (4) -- (5) -- (6) (7) -- (8) -- (9) (10) -- (11) -- (12) (13) -- (14) -- (15) (17) -- (18) -- (19) (20) -- (21);
\draw[ultra thick] (0) -- node[ulabela,xshift=0.7mm] {2} (1) -- node[ulabela] {1} (2) -- node[ulabela] {0} (3) (6) -- node[ulabela] {3} (7) (9) -- node[ulabelb] {5} (10) (12) -- node[ulabela] {7} (13) (15) -- node[ulabelb] {10} (16) -- node[ulabela] {8} (17) (21) -- node[ulabela] {12} (22);
\draw[line width=2pt, red] (19) -- node[ulabela] {11} (20);
\draw[dashed] (22) -- node[ulabelb] {9} (23) -- node[ulabelb] {6} (24) -- node[ulabela] {4} (25);

\node[anchor=west] at (0.8,-3.2) {$y=1110001001001001100101$};
\node[anchor=east] at (21.5,1) {$y=M_{n,k}^{11,\uparrow}(x)$};
\node at (24,3) {$y^{\diagup}$};
\node at (-5,0) {level $k+1=10$};
\end{scope}

\end{tikzpicture}
\caption{Definition of $p$-lexical matchings between levels~9 and~10 of~$Q_{22}$, where steps flipped along the $p$-lexical edge are marked with~$p$.
Between those two levels, the vertex~$x$ is incident with $p$-lexical edges for each $p\in\{0,1,\ldots,12\}$, and the vertex~$y$ is incident with $p$-lexical edges for each $p\in\{0,1,\ldots,12\}\setminus\{4,6,9\}$.
}
\label{fig:lexmatch}
\end{figure}

Lexical matchings in~$Q_n$ were introduced by Kierstead and Trotter~\cite{MR962224}, and they are parametrized by some integer $p\in\{0,1,\ldots,n-1\}$.
These matchings are defined as follows; see Figure~\ref{fig:lexmatch}.
We interpret a bitstring~$x$ as a lattice path, and we let~$x^\diagdown$ denote the lattice path that is obtained by appending $\diagdown$-steps to~$x$ until the resulting path ends at height~$-1$.
If $x$ ends at a height less than~$-1$, then~$x^\diagdown:=x$.
Similarly, we let~$x^{\,\diagup}$ denote the lattice path obtained by appending $\diagup$-steps to~$x$ until the resulting path ends at height~$+1$.
If $x$ ends at a height more than~$+1$, then~$x^{\,\diagup}:=x$.
We let $L_{n,k}$ denote the set of all vertices on level~$k$ of~$Q_n$, and we define a matching by two partial mappings $M_{n,k}^{p,\uparrow}\colon L_{n,k}\rightarrow L_{n,k+1}$ and $M_{n,k}^{p,\downarrow} \colon L_{n,k+1}\rightarrow L_{n,k}$ defined as follows:
For any~$x\in L_{n,k}$ we consider the lattice path~$x^\diagdown$ and scan it row-wise from top to bottom, and from right to left in each row.
The partial mapping~$M_{n,k}^{p,\uparrow}(x)$ is obtained by flipping the $p$th $\diagdown$-step encountered in this fashion, where counting starts with $0,1,\ldots$, if this $\diagdown$-step is part of the subpath~$x$ of~$x^\diagdown$; otherwise $x$ is left unmatched.
Similarly, for any~$x\in L_{n,k+1}$ we consider the lattice path~$x^{\,\diagup}$ and scan it row-wise from top to bottom, and from left to right in each row.
The partial mapping~$M_{n,k}^{p,\downarrow}(x)$ is obtained by flipping the $p$th $\diagup$-step encountered in this fashion if this~$\diagup$-step is part of the subpath~$x$ of~$x^{\,\diagup}$; otherwise $x$ is left unmatched.
It is straightforward to verify that these two partial mappings are inverse to each other, so they indeed define a matching between levels~$k$ and~$k+1$ of~$Q_n$, called the \emph{$p$-lexical matching}, which we denote by~$M_{n,k}^p$.
We also define $M^p_n:=\bigcup_{0\le k <n} M^p_{n,k}$, where we omit the index~$n$ whenever it is clear from the context.
In the following, we will only ever use $p$-lexical edges for $p=0,1,2$.
For instance, it is well-known that taking the union of all 0-lexical edges, i.e., the set $M^0$, yields exactly the Greene-Kleitman SCD~\cite{MR962224}.
This property is captured by the following lemma, together with several other explicit perfect matchings, consisting of $\{0,1,2\}$-lexical edges between certain sets of vertices from our SCD; see Figure~\ref{fig:match}.

To state the lemma, for a set $M$ of edges of~$Q_n$ and disjoint sets $A,B$ of vertices, we let $M[A,B]$ denote the set of all edges of~$M$ between $A$ and~$B$.
Moreover, for any vertex $x\in C^-_{h,i}$, $1<i<h\le n$, we consider the chain factorization $x=u_0 \,1\,\cdots u_{i-2}\,1\,u_{i-1}\,1\,0\,u_{i+1}\cdots \,0\,u_h$ with $u_0,\ldots,u_h\in D$, and we define a neighbor~$z(x)$ on the level below by
\begin{equation*}
z(x):=\begin{cases}
u_0\,1\cdots u_{i-2}\,1\,0\,0\,u_{i+1}\cdots 0\,u_h & \text{ if }u_{i-1}=\varepsilon, \\
u_0\,1\cdots u_{i-2}\,1\,0\,v\,0\,w\,1\,0\,u_{i+1}\cdots 0\,u_h & \text{ if }u_{i-1}=0\,v\,1\,w\text{ with }v,w\in D.
\end{cases}
\end{equation*}
Note that in the first case, $(x,z(x))$ is a 0-lexical edge, and in the second case, $(x,z(x))$ is a 2-lexical edge.
The result of this operation can be written more compactly as
\begin{equation}
\label{eq:z}
z(x)=u_0\,1\cdots u_{i-2}\,1\,0\,\rot{u_{i-1}}\,0\,u_{i+1}\cdots 0\,u_h.
\end{equation}

\begin{lemma}
\label{lem:lex}
For every $n\ge 3$, the following sets of edges~$M[A,B]$ are perfect matchings in~$Q_n$ between the vertex sets~$A$ and~$B$.
\begin{enumerate}[label=(\roman*)]
\item $M^0[C_{h,i},C_{h,i+1}]$ for every $0\le i < h \le n$;
\item $M^1[C^-_{1,0},C^-_{1,1}]$, $M^1[C^+_{h,i},C^-_{h+2,i+2}]$, and $M^1[C^+_{h,i},C^-_{h+2,i}]$ for every $0\le i \le h \le n-2$;
\item $Z^{02}[C^-_{h,i-1},C^-_{h,i}]$ for every $1<i<h\le n$, where $Z^{02}:=\{(x,z(x))\mid x\in C^-_{h,i}\}$.
\end{enumerate}
\end{lemma}

\begin{figure}[h!]
\includegraphics{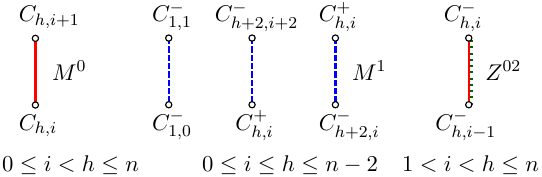}
\caption{Perfect matchings described by Lemma~\ref{lem:lex}.
The $\{0,1,2\}$-lexical edges are drawn solid, dashed, and dotted, respectively.
}
\label{fig:match}
\end{figure}

\begin{proof}
To prove~(i), let $x\in C_{h,i}$ and consider the chain factorization~\eqref{eq:xi} of~$x$.
By the definition of lexical matchings, the neighbor~$y$ of~$x$ on the level above reached via the 0-lexical edge is obtained by flipping the 0-bit after the valley~$u_i$ in~$x$, so $y=u_0 \,1\,\cdots u_{i}\,1\,u_{i+1}\,0\,u_{i+2}\cdots \,0\,u_h.$
From this and~\eqref{eq:xi} we conclude that $y\in C_{h,i+1}$ and that these edges reach all vertices in~$C_{h,i+1}$.

To prove~(ii), first consider the subcase $x\in C^-_{1,0}$ and the chain factorization $x=0\,u_1$ with $u_1\in D$.
Moreover, as $n\geq 3$, we have $u_1\neq \varepsilon$, so we may consider the right factorization $u_1=v\,0\,w\,1$ with $v,w\in D$.
By the definition of lexical matchings, the neighbor~$y$ of~$x=0\,u_1=0\,v\,0\,w\,1$ on the level above reached via the 1-lexical edge is obtained by flipping the 0-bit after the valley~$v$ in~$x$, so $y=0\,v\,1\,w\,1=u_0\,1$ with $u_0\in D$ defined by $\rot{u_0}=u_1$.
From this and~\eqref{eq:xi} we conclude that $y\in C^-_{1,1}$ and that these edges reach all vertices in~$C^-_{1,1}$.

We now consider the subcase $x\in C^+_{h,i}$ and the chain factorization~\eqref{eq:xi} of~$x$.
We know that $u_i\neq \varepsilon$, so $u_i$ has the right factorization $u_i=v\,0\,w\,1$ with $v,w\in D$.
The neighbor~$y$ of $x$ on the level above reached via the 1-lexical edge is obtained by flipping the 0-bit after the valley~$v$ in~$x$, so $y=u_0 \,1\,\cdots u_{i-1}\,1\,v\,1\,w\,1\,0\,u_{i+1}\cdots \,0\,u_h$.
From this and~\eqref{eq:xi} we conclude that $y\in C^-_{h+2,i+2}$ and that these edges reach all vertices in~$C^-_{h+2,i+2}$.
For the same~$x$, let us now compute the neighbor~$y'$ of~$x$ on the level below reached via the 1-lexical edge.
For this we consider the left factorization $u_i=0\,v'\,1\,w'$ with $v',w'\in D$.
The vertex~$y'$ is reached by flipping the 1-bit after the valley~$v'$ in~$x$, so $y'=u_0 \,1\,\cdots u_{i-1}\,1\,0\,v'\,0\,w'\,0\,u_{i+1}\cdots \,0\,u_h$.
Similarly to before, we obtain that $y'\in C^-_{h+2,i}$ and that these edges reach all vertices in~$C^-_{h+2,i}$, so the proof of part~(ii) is complete.

To prove part~(iii), note that \eqref{eq:xi} and~\eqref{eq:z} imply that $z(x)\in C^-_{h,i-1}$ and that $C^-_{h,i-1}$ is precisely the image of $C^-_{h,i}$ under the mapping~$z$.
\end{proof}

\section{Cycle factor construction}
\label{sec:cfac}

We now construct a cycle factor~$\cC_{n,\ell}$ in the graph~$Q_{n,\ell}$, $n=2m+1$, i.e., in the subgraph of the $n$-cube induced by the middle $2\ell$ levels.
Throughout this and the following sections we consider fixed~$m\geq 1$ and $2\le \ell\le m+1$.
We construct the cycle factor incrementally, starting with chains from the Greene-Kleitman SCD and adding $\{0,1,2\}$-lexical edges between certain sets of vertices, see Figure~\ref{fig:factor}.
In the following, when referring to a subgraph given by a set of edges, we mean the subgraph of~$Q_{n,\ell}$ induced by those edges.
Moreover, we say that a chain is \emph{short} if its length is at most $2\ell-3$, i.e., if it does not span all levels of~$Q_{n,\ell}$.

Our construction starts by taking all those short chains, formally
\begin{subequations}
\label{eq:X}
\begin{equation}
X^0:=\bigcup_{0\le i<h\le 2\ell-3}M^0[C_{h,i},C_{h,i+1}];
\end{equation}
recall Lemma~\ref{lem:lex}~(i).
From Lemma~\ref{lem:lex}~(ii) we know that 1-lexical edges perfectly match all bottom vertices of $[-+]$-chains of length~1 with all top vertices of $[+-]$-chains of length~1 along the edges
\begin{equation}
\Xm^1:=M^1[C^-_{1,0},C^-_{1,1}].
\end{equation}
Furthermore, for $1\le h\le 2\ell-5$, 1-lexical edges perfectly match all top vertices of $[?+]$-chains of length~$h$ with all top vertices of $[?-]$-chains of length~$h+2$, and all bottom vertices of $[+?]$-chains of length~$h$ with all bottom vertices of $[-?]$-chains of length~$h+2$ along the edges
\begin{equation}
\Xt^1:=\bigcup_{1\le h\le 2\ell-5} M^1[C^+_{h,h},C^-_{h+2,h+2}], \quad
\Xb^1:=\bigcup_{1\le h\le 2\ell-5} M^1[C^+_{h,0},C^-_{h+2,0}],
\end{equation}
respectively.
Note that the only vertices of short chains that have degree~1 in the set
\begin{equation}
X:=X^0 \cup \Xm^1 \cup \Xt^1 \cup \Xb^1
\end{equation}
\end{subequations}
are exactly the vertices of $C^+_{2\ell-3,2\ell-3}$ and $C^+_{2\ell-3,0}$; that is, the top vertices of $[?+]$-chains of length~$2\ell-3$ and the bottom vertices of $[+?]$-chains of length~$2\ell-3$.

\begin{figure}
\includegraphics{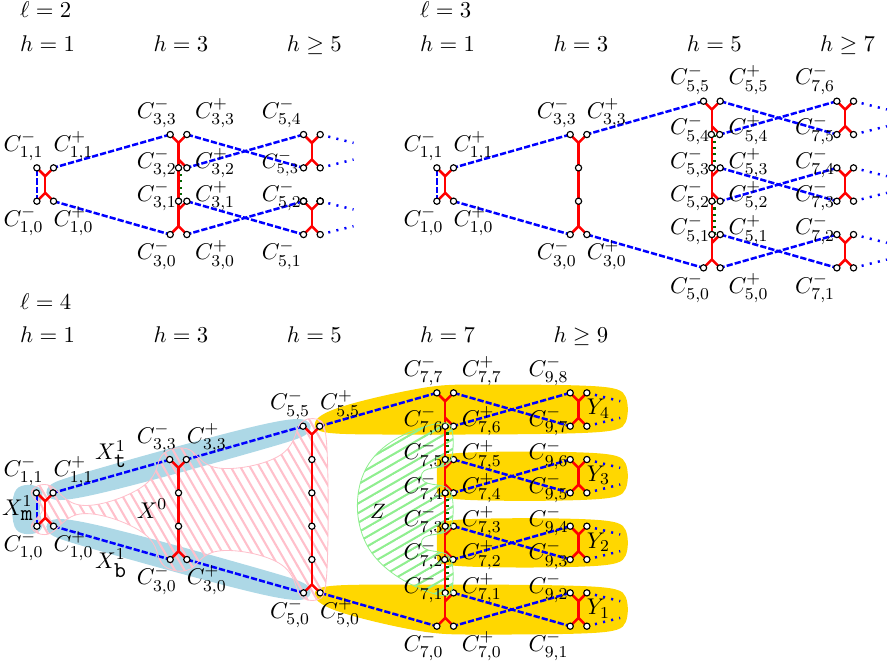}
\caption{Illustration of the cycle factor $\cC_{n,\ell}$ for $\ell=2,3,4$.
Each bullet represents an entire set of vertices, as specified in the figure, lines between them specify perfect matchings between these sets.
The $\{0,1,2\}$-lexical edges are drawn with solid, dashed, and dotted lines, respectively.
In the bottom part, various sets of matching edges are highlighted.
}
\label{fig:factor}
\end{figure}

Next, between every pair of consecutive levels of $Q_{n,\ell}$ we take all 0-lexical and 1-lexical edges that are not incident to a degree-2 vertex in~$X$.
Specifically, between these pairs of levels we take all 0-lexical edges from chains that are not short and all 1-lexical edges between chains that are not short.
In addition, between the top two levels we take all 1-lexical edges between top vertices of $[?+]$-chains of length~$2\ell-3$ and top vertices of $[?-]$-chains of length~$2\ell-1$, and symmetrically, between the bottom two levels we take all 1-lexical edges between bottom vertices of $[+?]$-chains of length~$2\ell-3$ and bottom vertices of $[-?]$-chains of length~$2\ell-1$.
Formally, these sets of edges are
\begin{subequations}
\label{eq:Y}
\begin{equation}
Y_1:=Y'_1 \cup M^1[C^+_{2\ell-3,0},C^-_{2\ell-1,0}], \quad Y_{\ell}:=Y'_{\ell} \cup M^1[C^+_{2\ell-3,2\ell-3},C^-_{2\ell-1,2\ell-1}], \quad
Y_k:=Y'_k
\end{equation}
for $1<k<\ell$ where
\begin{equation}
Y'_k := \bigcup_{\substack{h\ge 2\ell-1 \\ i:=(h-(2\ell-1))/2+2(k-1)}} M^0[C_{h,i},C_{h,i+1}] \cup M^1[C^+_{h,i},C^-_{h+2,i+2}] \cup M^1[C^+_{h,i+1},C^-_{h+2,i+1}]
\end{equation}
for $1\le k \le \ell$.
Note that $Y_1$ and $Y_\ell$ contain all $\{0,1\}$-lexical edges between the bottom two levels or the top two levels of~$Q_{n,\ell}$, respectively.
We also define
\begin{equation}
Y:=\bigcup_{1\le k \le \ell} Y_k.
\end{equation}
\end{subequations}
As a consequence of these definitions and Lemma~\ref{lem:lex}~(i) and~(ii), the only vertices of $Q_{n,\ell}$ that have degree~1 in the set $X\cup Y$ are exactly the vertices of $C^-_{2\ell-1,i}$ for $1\le i\le 2\ell-2$.
We thus add the edges
\begin{equation}
\label{eq:Z}
Z:=\bigcup_{i=1,3,5,\ldots,2\ell-3} Z^{02}[C^-_{2\ell-1,i},C^-_{2\ell-1,i+1}]
\end{equation}
defined in part~(iii) of Lemma~\ref{lem:lex}, which makes
\begin{equation}
\label{eq:cfac}
\cC_{n,\ell}:=X\cup Y\cup Z
\end{equation}
a cycle factor in the graph~$Q_{n,\ell}$.

Note that if $\ell=2$, then the sets $\Xt^1$ and $\Xb^1$ are empty and $X^0$ contains only chains of length~1; see the top left part of Figure~\ref{fig:factor}.
In the other extreme case $\ell=m+1$, the set $Y\cup Z$ contains only a single path of length~$n+2$, namely the unique chain of length~$2\ell-1=n$ with an additional 1-lexical edge from~$Y_1$ and~$Y_{\ell}$ attached on each side.

\subsection{Comparison with previous constructions}
\label{sec:comparison}

Our cycle factor construction generalizes the construction for~$\ell=1$ presented in \cite{MR3483129,gregor-muetze-nummenpalo:18}, which simply consisted in taking the union of all 0-lexical and 1-lexical edges between the middle two levels.
It also generalizes the construction for~$\ell=2$ presented in~\cite{jaeger-et-al-journal:21}, which also only used $\{0,1,2\}$-lexical matchings.
In fact, all these earlier papers actually used $\{m,m-1,m-2\}$-lexical matching edges, but these are isomorphic to $\{0,1,2\}$-lexical edges by reversing bitstrings.
The earlier construction for~$\ell=2$ seemed rather arbitrary at the time, but now nicely fits into the general picture shown in Figure~\ref{fig:factor}\footnote{As the picture of this construction resembles a rocket, with the tip on the left and the boosters on the right, one might be tempted to consider this rocket science.}.

\section{Structure of cycles}
\label{sec:struc}

In this section, we describe the structure of cycles in the factor~$\cC_{n,\ell}$ defined in~\eqref{eq:cfac} (where $n=2m+1$, $m\geq 1$, $2\le \ell\le m+1$).
In particular, we give a combinatorial interpretation for certain vertices encountered along each cycle, allowing us to compute the number and lengths of some of the cycles.
In the following, we call a cycle that contains a short $[--]$-chain from the Greene-Kleitman SCD a \emph{short cycle}, and any other cycle is called \emph{long}.
The key properties about short and long cycles are captured in Lemmas~\ref{lem:short} and~\ref{lem:long} below, respectively.

\subsection{Short cycles}
\label{sec:short}

The next lemma completely describes the structure of short cycles; see Figure~\ref{fig:short}.
Note that for $\ell=2$ there are no short $[--]$-chains and hence no short cycles.

\begin{lemma}
\label{lem:short}
The short cycles in the factor~$\cC_{n,\ell}$ defined in~\eqref{eq:cfac} satisfy the following properties:
\begin{enumerate}[label=(\roman*)]
\item For every $[--]$-chain $C={*}\,u_1\,{*}\,u_2\,{*}$ of length~$h=3\le 2\ell-3$, the short cycle containing~$C$ also contains the $[+-]$-chain $a\,{*}$ of length~1 and the $[-+]$-chain ${*}\,\rot{a}$ of length~1, where $a:=0\,u_1\,1\,u_2\in D$, plus three additional edges connecting these chains.
\item For every $[--]$-chain $C=*\,u_1\,{*}\cdots u_{h-1}\,{*}$ of length~$5\le h \le 2\ell-3$, the short cycle containing~$C$ also contains the $[+-]$-chain $a*u_3*\cdots u_{h-1}\,{*}$ of length~$h-2$, the $[-+]$-chain ${*}\,u_1\,{*}\cdots u_{h-3}\,{*}\,b$ of length~$h-2$, and the $[++]$-chain $a\,{*}\,u_3\,{*}\cdots u_{h-3}\,{*}\,b$ of length~$h-4$, where $a:=0\,u_1\,1\,u_2\in D$ and $b:= u_{h-2}\,0\,u_{h-1}\,1\in D$, plus four additional edges connecting these chains.
\end{enumerate}
In particular, all short cycles lie entirely within the edge set~$X$ defined in~\eqref{eq:X}, and the only edges in~$X$ not contained in a short cycle are $[++]$-chains of length~$2\ell-5$ if~$\ell\geq 3$ and $[+-]$-, $[-+]$-, and $[++]$-chains of length~$2\ell-3$.
\end{lemma}

\begin{figure}
\makebox[0cm]{ 
\includegraphics{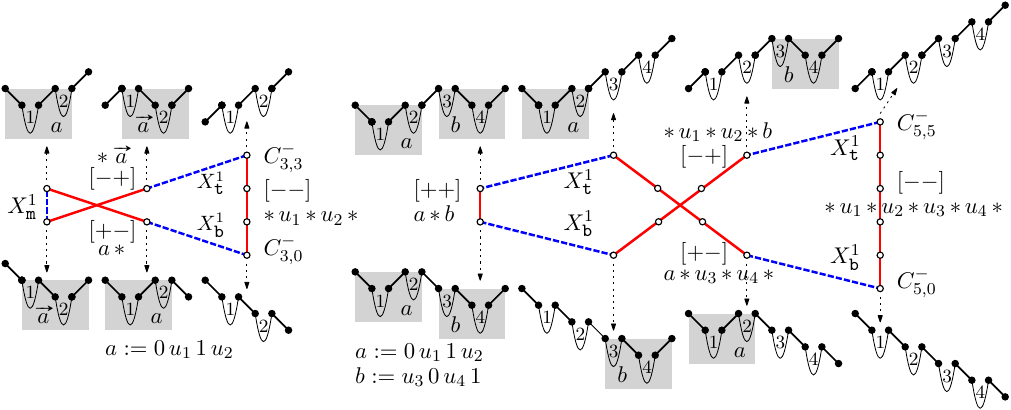}
}
\caption{Structure of short cycles as described by Lemma~\ref{lem:short}, for case~(i) on the left and for case~(ii) on the right.
Valleys~$u_i$ are labeled by~$i$ in the lattice paths for better readability.
\label{fig:short}
}
\end{figure}

\begin{proof}
To prove parts~(i) and~(ii) of the lemma consider Figure~\ref{fig:short}, and observe that edges of $\Xm^1\cup \Xt^1\cup \Xb^1$ connect end vertices of $[--]$-chains as in parts~(i) or~(ii) of the lemma into cycles of the claimed form.
For this recall the definition~\eqref{eq:X} and use the definition of 1-lexical matchings.
The last part of the lemma follows by observing that short cycles as described in part~(i) and (ii) pick up all short chains except the ones mentioned in the lemma.
\end{proof}

Note that by Lemma~\ref{lem:short}, the short cycle containing a $[--]$-chain of length~$h$ has total length~$4h-4$.

We say that a cycle in $\cC_{n,\ell}$ has \emph{range} $2r$ if it is contained in the middle $2r$ levels but not in the middle $2r-2$ levels.
Clearly, the short cycle containing a $[--]$-chain of length~$h$ has range $2r=h+1\geq 4$, and it visits vertices in all $2r$ middle levels.
We will see later that long cycles have range~$2\ell$, and that each long cycle visits vertices in all $2\ell$ levels.
The next corollary is an immediate consequence of Lemma~\ref{lem:short}.

\begin{corollary}
\label{cor:range}
For fixed~$n$, any short cycle of range~$2r$ appears in each of the cycle factors $\cC_{n,\ell}$ for $\ell>r$.
\end{corollary}

For any $n=2m+1$, $m\geq 1$, we can compute the number of $[--]$-chains of length~$h$ (by counting lattice paths using a reflection trick), and we thus obtain the number of short cycles of range~$2r$ as
\[\binom{2m-1}{m-r+1}-\binom{2m-1}{m-r}=\frac{r-1}{m}\binom{2m}{m-r+1},\]
see Table~\ref{tab:2factors}.
As we shall see, the structure of long cycles is more complicated in general, and we are not able to count them explicitly, with few exceptions:
For $\ell=1$ the number of all cycles of the factor is given by the number of plane trees with $m$ edges (see \cite[Proposition~2]{gregor-muetze-nummenpalo:18}).
For $\ell=2$ the number of all (long) cycles is given by the number of plane trivalent trees with $m$ internal vertices (see \cite[Proposition~12]{jaeger-et-al-journal:21}).
For $\ell=m+1$ there is exactly one long cycle, and so the total number of cycles in the factor is~$\binom{2m-1}{m-1}$.

\begin{table}
\caption{Number of cycles in the factor~$\cC_{n,\ell}$, $n=2m+1$, for small values of~$m$ and~$\ell$.
The first number in brackets counts short cycles of range $2\ell-2$, the second number counts long cycles (of range~$2\ell$).
For fixed~$m$ and~$\ell$, the full distribution of cycles across ranges can be recovered by considering all other table entries on the same row.
For instance, the full distribution of cycles of all ranges for $m=8$ and $\ell=9$ is $[1430,2002,1638,910,350,90,14,1]$ from 1430 short cycles of range~4 (through the middle 4~levels of the 17-cube), to 14 short cycles of range~16, to a single long cycle of range~18 (through all levels).
}
\footnotesize
\renewcommand\tabcolsep{2pt}
\begin{tabular}{l|rrrrrrrrrrrrr}
$m$ & $\ell=1$ & 2 & 3 & 4 & 5 & 6 & 7 & 8 & 9 \\ \hline
1  & 1 & 1 [0,1] &  &  &  &  &  &   \\
2  & 1 & 1 [0,1] & 3 [2,1] &  &  &  &  &  &   \\
3  & 2 & 1 [0,1] & 6 [5,1] & 10 [4,1] &  &  &  &  &  \\
4  & 3 & 4 [0,4] & 17 [14,3] & 29 [14,1] & 35 [6,1] &  &  &  &  \\
5  & 6 & 6 [0,6] & 46 [42,4] & 93 [48,3] & 118 [27,1] & 126 [8,1] &  &  &  \\
6  & 14 & 19 [0,19] & 142 [132,10] & 307 [165,10] & 412 [110,5] & 452 [44,1] & 462 [10,1] &  &  \\
7  & 34 & 49 [0,49] & 446 [429,17] & 1010 [572,9] & 1438 [429,8] & 1643 [208,5] & 1704 [65,1] & 1716 [12,1] &     \\
8  & 95 & 150 [0,150] & 1475 [1430,45] & 3474 [2002,42] & 5113 [1638,43] & 6002 [910,22] & 6337 [350,7] & 6421 [90,1] & 6435 [14,1]  \\
\end{tabular}
\label{tab:2factors}
\end{table}

\subsection{Long cycles}
\label{sec:long}

We now describe long cycles.
First we show that each of the sets~$Y_k$ defined in~\eqref{eq:Y} contains only paths, but no cycles, which will allow us to show that every long cycle has range~$2\ell$ and visits vertices from all~$2\ell$ levels.
To describe the end vertices of the paths formed by the edge set~$Y_k$, we use the following result shown in~\cite[Proposition~2]{gregor-muetze-nummenpalo:18}.

\begin{lemma}
\label{lem:paths}
For any~$d\geq 1$, the union of the 0- and 1-lexical matchings between levels~$d$ and~$d+1$ of~$Q_{2d}$ contains no cycles, but only paths, and the sets of first and last vertices of these paths are~$D_{2d}$ and~$D'_{2d}$, respectively.
Furthermore, for any path with first vertex $x\in D_{2d}$ and last vertex $y\in D'_{2d}$, if $x=u\,0\,v\,1$, with $u,v\in D$, is the right factorization of~$x$, then $y=u\,1\,0\,v$.
\end{lemma}

\begin{lemma}
\label{lem:Ypaths}
For every $1\le k \le \ell$, the edge set~$Y_k$ defined in~\eqref{eq:Y} contains only paths, but no cycles.
Furthermore, any path formed by the edges of~$Y_k$ has its first and last vertex in the following sets, and if its first vertex~$x$ has the form specified below, then its last vertex~$y$ has the form specified below:
\begin{enumerate}[label=(\roman*)]
\item If $k=\ell$, then we have $x\in C^+_{2\ell-3,2\ell-3}$ and $y\in C^-_{2\ell-1,2\ell-2}$, and if
\begin{align*}
x&=u_0\,1\cdots u_{2\ell-4}\,1\,\underbrace{v\,0\,w\,1}_{u_{2\ell-3}} \in C^+_{2\ell-3,2\ell-3}, \text{ then } \\
y&=u_0\,1\cdots u_{2\ell-4}\,1\,v\,1\,0\,w\,\in C^-_{2\ell-1,2\ell-2}.
\end{align*}
\item If $k=1$, then we have $x\in C^+_{2\ell-3,0}$ and $y\in C^-_{2\ell-1,1}$, and if
\begin{align*}
x&=\underbrace{0\,v\,1\,w}_{u_0}\,0\,u_1\cdots 0\,u_{2\ell-3} \in C^+_{2\ell-3,0}, \text{ then} \\
y&=v\,1\,0\,w\,0\,u_1\cdots 0\,u_{2\ell-3} \in C^-_{2\ell-1,1}.
\end{align*}
\item If $1<k<\ell$, then we have $x\in C^-_{2\ell-1,2k-1}$ and $y\in C^-_{2\ell-1,2k-2}$, and if
\begin{align*}
x&=u_0\,1\cdots u_{2k-3}\,1\,u_{2k-2}\,1\,0\,u_{2k} \cdots 0\,u_{2\ell-1} \in C^-_{2\ell-1,2k-1}, \text{ then} \\
y&=u_0\,1\cdots u_{2k-3}\,1\,0\,u_{2k-2}\,0\,u_{2k} \cdots 0\,u_{2\ell-1} \in C^-_{2\ell-1,2k-2}.
\end{align*}
\end{enumerate}
\end{lemma}

\begin{proof}
For any bitstring~$x=x_1\cdots x_n$ and any integers~$\alpha\geq 0$ and $0\le \beta\le n$ we define the mapping $\varrho_{\alpha,\beta}(x):=\varrho_{\alpha,\beta}(x_1\cdots x_n):=0^\alpha x_1\cdots x_{n-\beta}$, i.e., this mapping prepends~$x$ with $\alpha$ many 0s, and removes the last $\beta$ bits.

To prove~(i), we define $\alpha:=2\ell-3$ and we consider the edge set~$\varrho_{\alpha,0}(Y_{\ell})$ between levels~$d$ and~$d+1$ of~$Q_{2d}$, where $2d=n+\alpha$.
From the definition of lexical matchings and~\eqref{eq:Y}, we see that $\varrho_{\alpha,0}(Y_{\ell})$ is a set of 0- and 1-lexical matching edges in~$Q_{2d}$.
Moreover, as $\varrho_{\alpha,0}(C^+_{2\ell-3,2\ell-3})\seq D_{2d}$ and $\varrho_{\alpha,0}(C^-_{2\ell-1,2\ell-2})\seq D'_{2d}$, applying Lemma~\ref{lem:paths} in~$Q_{2d}$ shows that the edge set~$Y_\ell$ contains no cycles and any path formed by those edges has its first and last vertex in the sets~$C^+_{2\ell-3,2\ell-3}$ and~$C^-_{2\ell-1,2\ell-2}$.
Moreover, for a first vertex $x':=0^\alpha\,x\in D_{2d}$ with $x$ as in the lemma we have the right factorization $x'=0^\alpha\,x=a\,0\,w\,1$ with $a:=0^\alpha\,u_0\,1\cdots u_{2\ell-4}\,1\,v\in D$, so Lemma~\ref{lem:paths} shows that the last vertex of this path is indeed $y'=a\,1\,0\,w=0^\alpha\,y$ with $y$ as above.

To prove~(ii), consider the automorphism $\varphi$ of~$Q_n$ that reverses and complements all bits.
One can check that $\varphi$ maps 0- and 1-lexical matchings between levels~$i$ and~$i+1$ to 0- and 1-lexical matchings between levels~$n-i$ and $n-i-1$ for all $0\le i<n$.
Consequently, using~\eqref{eq:Y} we obtain that $Y_1=\varphi(Y_{\ell})$.
As moreover $C^+_{2\ell-3,0}=\varphi(C^+_{2\ell-3,2\ell-3})$ and $C^-_{2\ell-1,1}=\varphi(C^-_{2\ell-1,2\ell-2})$ the claim follows from part~(i).

To prove~(iii), let $x\in C^-_{2\ell-1,2k-1}$ be as in the lemma.
Moreover, let $b:=0\,u_{2k} \cdots 0\,u_{2\ell-1}$ be the suffix of~$x$, let $\beta$ be the length of~$b$, and let $\alpha:=2k-3$.
Consider the vertex $z:=u_0\,1\cdots u_{2k-3}\,0\,u_{2k-2}\,1\,b$, which is joined to~$x$ via a 1-lexical edge.
This edge, however, is not in~$Y_k$, as we have $z\in C^+_{2\ell-3,2k-3}$ from~\eqref{eq:xi}.
Adding all those edges to~$Y_k$ yields a larger set of edges $Y':=Y_k\cup M^1[C^+_{2\ell-3,2k-3},C^-_{2\ell-1,2k-1}]$.
We now consider the edge set $\varrho_{\alpha,\beta}(Y')$ between levels~$d$ and~$d+1$ of~$Q_{2d}$, where $2d=n+\alpha-\beta$.
Similarly to before, $\varrho_{\alpha,\beta}(Y')$ is a set of 0-lexical and 1-lexical matching edges in~$Q_{2d}$.
Moreover, as $\varrho_{\alpha,\beta}(C^+_{2\ell-3,2k-3})\seq D_{2d}$ and $\varrho_{\alpha,\beta}(C^-_{2\ell-1,2k-2})\seq D'_{2d}$, applying Lemma~\ref{lem:paths} in~$Q_{2d}$ shows that the edge set~$Y'$, and hence~$Y_k$, contains no cycles and any path formed by the edges of~$Y'$ has its first and last vertex in the sets~$C^+_{2\ell-3,2k-3}$ and~$C^-_{2\ell-1,2k-2}$.
Moreover, for the first vertex $z':=\varrho_{\alpha,\beta}(z)\in D_{2d}$ with $z$ from before we have the right factorization $z'=a\,0\,u_{2k-2}\,1$ with $a:=0^\alpha\,u_0\,1\,\cdots u_{2k-3}\in D$, so Lemma~\ref{lem:paths} shows that the last vertex of this path is $y':=a\,1\,0\,u_{2k-2}$, and as $y'=\varrho_{\alpha,\beta}(y)$ with $y$ as in the lemma, the claim is proved.
\end{proof}

The next lemma describes the effect of walking along a path in the set $Y\cup Z$ from its first to its last vertex.

\begin{lemma}
\label{lem:YZpaths}
The union of edges $Y\cup Z$ with $Y$ and~$Z$ as defined in~\eqref{eq:Y} and~\eqref{eq:Z} contains only paths, but no cycles, and the sets of first and last vertices of these paths are~$C^+_{2\ell-3,2\ell-3}$ and~$C^+_{2\ell-3,0}$, respectively.
Furthermore, for any path with first vertex
\[x=u_0\,1\,u_1 \cdots 1\,\underbrace{v\,0\,w\,1}_{u_{2\ell-3}} \in C^+_{2\ell-3,2\ell-3},\]
the corresponding last vertex is
\[y=0\,u_0\,1\,\rot{u_1}\,0\,u_2\,0\,\rot{u_3}\,0\,u_4\,0 \cdots \rot{u_{2\ell-5}}\,0\,u_{2\ell-4}\,0 \rot{v}\,0\,w \in C^+_{2\ell-3,0}.\]
\end{lemma}

\begin{proof}
As the edges in~$Z$ perfectly match the last vertices of paths formed by~$Y_k$ with the first vertices of paths formed by~$Y_{k-1}$ for all $2<k \le\ell$, and the last vertices of paths formed by~$Y_2$ with the last vertices of paths formed by~$Y_1$, this follows by iterating Lemma~\ref{lem:Ypaths} and by using~\eqref{eq:z}, where the $Z$-edges cause the rotation of every second valley $u_1,u_3,\dots,u_{2\ell-5},v$ in~$x$.
\end{proof}

\begin{figure}
\makebox[0cm]{ 
\includegraphics{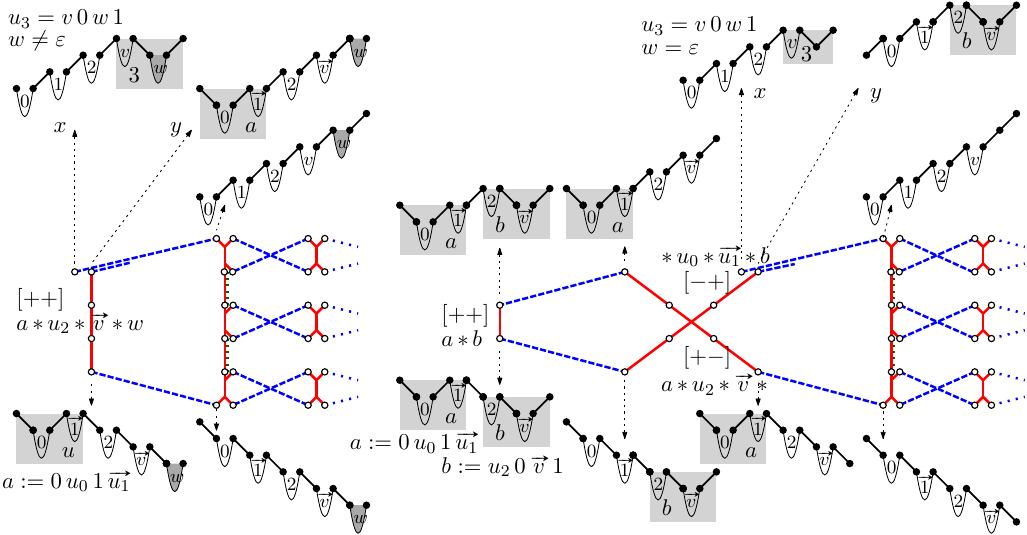}
}
\caption{Structure of long cycles as described by Lemma~\ref{lem:long}, for case~(i) ($w\ne \varepsilon$) on the left and for case~(ii) ($w=\varepsilon$) on the right.
\label{fig:long}
}
\end{figure}

The next lemma describes how the paths formed by the edges~$Y\cup Z$ interact with the short chains in~$X$ that are not contained in a short cycle (recall the last part of Lemma~\ref{lem:short}) to form long cycles.
We describe this interaction locally, in terms of how consecutive vertices from the set~$C^+_{2\ell-3,2\ell-3}$ on a long cycle look like, which is enough for our purposes.

\begin{lemma}
\label{lem:long}
Each long cycle in the factor~$\cC_{n,\ell}$ defined in~\eqref{eq:cfac} has range~$2\ell$ and visits vertices from all~$2\ell$ levels.
Moreover, let $x=u_0\,1\,u_1 \cdots 1\,u_{2\ell-4}\,1\,\underbrace{v\,0\,w\,1}_{u_{2\ell-3}} \in C^+_{2\ell-3,2\ell-3}$ be a vertex from the cycle and let~$y$ be the next vertex from $C^+_{2\ell-3,2\ell-3}$ on the cycle after~$x$.
\begin{enumerate}[label=(\roman*)]
\item If $w\ne\varepsilon$, then we have $y=a\,1\,u_2\,1\,\rot{u_3}\,1\,u_4\,1 \cdots \rot{u_{2\ell-5}}\,1\,u_{2\ell-4}\,1\,\rot{v}\,1\,w$, where $a:=0\,u_0\,1\,\rot{u_1}$, and between~$x$ and~$y$ the cycle traverses the $[++]$-chain $a\,{*}\,u_2\,{*}\,\rot{u_3}\,{*}\,u_4\,{*}\cdots \rot{u_{2\ell-5}}\,{*}\,u_{2\ell-4}\,{*}\,\rot{v}\,{*}\,w$ of length~$2\ell-3$.
\item If $w=\varepsilon$, then we have $y=1\,u_0\,1\,\rot{u_1}\,1\,u_2\,1\,\rot{u_3}\,1\cdots \rot{u_{2\ell-5}}\,1\,b$, where $b:=u_{2\ell-4}\,0\,\rot{v}\,1$, and between~$x$ and~$y$ the cycle traverses the $[+-]$-chain $a\,{*}\,u_2\,{*}\,\rot{u_3}\,{*}\cdots \rot{u_{2\ell-5}}\,{*}\,u_{2\ell-4}\,{*}\,\rot{v}\,{*}$ of length~$2\ell-3$, the $[++]$-chain $a\,{*}\,u_2\,{*}\,\rot{u_3}\,{*}\cdots \rot{u_{2\ell-5}}\,{*}\,b$ of length~$2\ell-5$, and the $[-+]$-chain ${*}\,u_0\,{*}\,\rot{u_1}\,{*}\,u_2\,{*}\cdots\rot{u_{2\ell-5}}\,{*}\,b$ of length~$2\ell-3$, where $a:=0\,u_0\,1\,\rot{u_1}$, plus two additional edges connecting these chains.
\end{enumerate}
\end{lemma}

\begin{proof}
The first part of the lemma follows immediately from the last part of~Lemma~\ref{lem:short} and the first part of Lemma~\ref{lem:YZpaths}.
To prove parts~(i) and (ii) consider Figure~\ref{fig:long}, and observe that a path in~$Y\cup Z$ starting from $x$ as in the lemma has last vertex as specified by Lemma~\ref{lem:YZpaths}, and then the cycle continues with a path in~$X$ on short chains as claimed until it arrives at the claimed vertex~$y$ (use~\eqref{eq:X} and the definition of 1-lexical matchings).
\end{proof}

\begin{figure}
\includegraphics{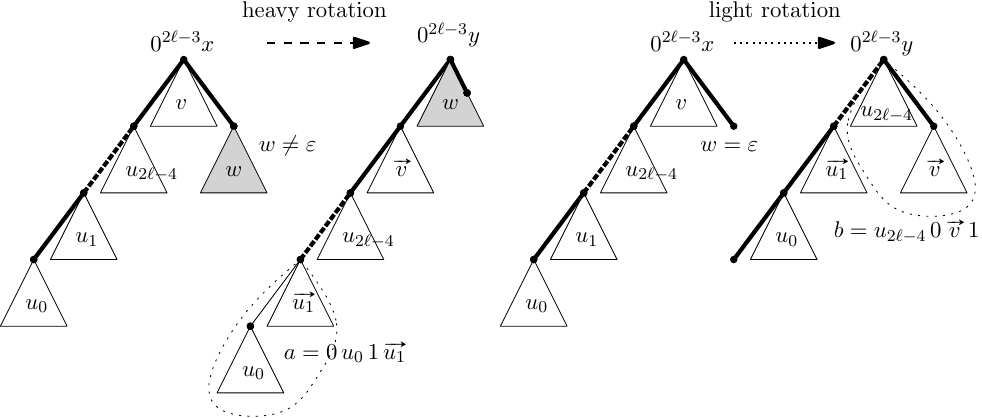}
\caption{Interpretation of the transformations stated in parts~(i) and (ii) of Lemma~\ref{lem:long} in terms of rooted trees, called heavy and light rotations, shown on the left and right hand side, respectively.
The spine edges are drawn bold.
}
\label{fig:hlrot}
\end{figure}

It is illuminating to interpret the transformations captured by parts~(i) and (ii) of Lemma~\ref{lem:long} in terms of rooted trees; see Figure~\ref{fig:hlrot}.
If we prepend $0^{2\ell-3}$ to the vertices~$x$ and~$y$ as in the lemma, then the resulting bitstrings $0^{2\ell-3}C^+_{2\ell-3,2\ell-3}=:\cT_{n,\ell}$ can be interpreted as rooted trees on $m+\ell-1$ edges with a root of degree at least~2 and the leftmost leaf in depth at least~$2\ell-3$.
We refer to the operations $0^{2\ell-3}x\mapsto 0^{2\ell-3}y$ with~$x$ and~$y$ as in parts~(i) or~(ii) of Lemma~\ref{lem:long} as a \emph{heavy rotation}, or a \emph{light rotation}, respectively.
Moreover, we refer to the subtree $0^{2\ell-3}\,1^{2\ell-3}\,0\,1$ of any tree from~$\cT_{n,\ell}$ as its \emph{spine}.
Combinatorially, a heavy rotation is an inverse rotation at the root, plus a rotation of every second subtree attached to the spine.
A light rotation moves all subtrees attached to the spine one spine vertex to the right, and it also applies a rotation to every second such subtree.
These operations define an equivalence relation on~$\cT_{n,\ell}$ whose equivalence classes correspond to long cycles.
Unfortunately, in general there seems to be no `nice' combinatorial interpretation of these equivalence classes, unless in some special cases like $\ell\in\{1,2,m+1\}$; recall the remarks from the end of Section~\ref{sec:short}.
The number of long cycles determined experimentally for small values of~$m$ and~$\ell$ can be found in Table~\ref{tab:2factors}.

\section{Flipping 4-cycles}
\label{sec:flip4}

In this section we describe how to modify the cycle factor~$\cC_{n,\ell}$ (where $n=2m+1$, $m\geq 1$, $2\le \ell\le m+1$), so that every short cycle is joined to some long cycle.
The remaining task, solved in the next sections, will then be to join the long cycles to a single Hamilton cycle.
The modifications via 4-cycles exploit relations between pairs of short chains of the Greene-Kleitmann SCD that were first used by Griggs, Killian, and Savage~\cite{MR2034416}, and by Killian, Ruskey, Savage, and Weston~\cite{MR2114190} with the purpose of constructing symmetric Venn diagrams.

A \emph{flipping $4$-cycle} between two vertex-disjoint paths $P,P'$ is a $4$-cycle $F$ that shares exactly one edge with each of the two paths.
Note that the symmetric difference of the edge sets $(P\cup P')\bigtriangleup F$ gives path $Q,Q'$ on the same vertex set with flipped end vertices:
Specifically, if $P,P'$ are $xy,x'y'$-paths, then $Q,Q'$ are $xy',x'y$-paths.
Thus, if $P,P'$ are subpaths of two cycles $C,C'$, then $(C\cup C')\bigtriangleup F$ is a single cycle.

\begin{figure}
\includegraphics{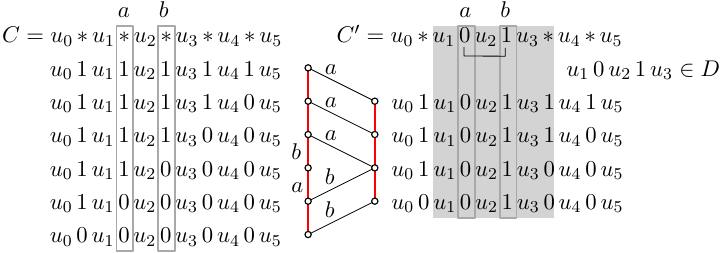}
\caption{Flipping $4$-cycles between a chain~$C$ and one of its children~$C'$.
The labels~$a$ and~$b$ along edges denote positions of flipped bits.
\label{fig:4cyc}
}
\end{figure}

Given a chain $C$ of length~$h$, we say that a chain~$C'$ of length~$h-2$ is a \emph{child} of $C$, if $C'$ is obtained from~$C$ by replacing two consecutive *s in~$C$ by~0 and~1, respectively; see Figure~\ref{fig:4cyc}.
We refer to this as \emph{matching two *s in~$C$}.
We also call~$C$ a \emph{parent} of~$C'$.
Note that $C$ has exactly $h-1$ many children, obtained by matching any two consecutive of the $h$ many *s in total.
Moreover, the number of parents of~$C'$ is given by the number of outermost matched 01-pairs in~$C'$, which can be up to~$(n-(h-2))/2$.
A straightforward verification (see Figure~\ref{fig:4cyc} for an example) yields the following lemma.

\begin{lemma}
\label{lem:4cyc}
For every chain~$C$ of length~$h$ and its child~$C'$ obtained by matching two *s at positions $a,b$, there are exactly $h-2$ many flipping $4$-cycles between~$C$ and~$C'$, each using a distinct edge of~$C$ and~$C'$, except the two consecutive edges of~$C$ that flip the coordinates~$a$ and~$b$.
\end{lemma}

The preceding lemma provides us enough options for selecting multiple edge-disjoint flipping 4-cycles as follows.

\begin{lemma}
\label{lem:4cyc-local}
For every chain~$C$ of length~$h\ge 5$ and any edge~$f$ of $C$, there are $h-1$ edge-disjoint flipping $4$-cycles between~$C$ and each of its $h-1$ children that all avoid the edge~$f$.
\end{lemma}

We can think of~$f$ as a forbidden edge that will be used to connect~$C$ further to one of its parent chains.

\begin{proof}
We may assume w.l.o.g.\ that $C=*^h$, as any other chain of length~$h$ is obtained from this chain by inserting valleys between and around the~*s; recall~\eqref{eq:C}.
We label the edges on~$C$ by $1,\ldots,h$ from bottom to top, and we let $p\in \{1,\ldots,h\}$ be the index of the edge~$f$ to be avoided.
We label all $h-1$ children of~$C$ from~1 to~$h-1$, such that the $i$th child is obtained by matching the $i$th and $(i+1)$st $*$ in~$C$.
To the $i$th child, $1\le i\le h-1$, we assign the edge
\[e(i):= \begin{cases}
\mu(i+2,h) & \text{if } \mu(i+2,h)<p,\\
\mu(i+3,h) & \text{otherwise},
\end{cases}\]
on~$C$, where $\mu(x,h)$ is the representative of the residue class of~$x$ modulo~$h$ from the set $\{1,\ldots,h\}$.
Since $e(i)\notin \{i,i+1\}$ by the definition of~$e(i)$ and the assumption~$h\geq 5$, by Lemma~\ref{lem:4cyc} there is a flipping 4-cycle between~$C$ and its $i$th child that uses the edge~$e(i)$ on~$C$.
Moreover, as the mapping $e \colon \{1,\dots, h-1\} \to \{1, \dots, p-1,p+1,\dots, h\}$ is a bijection, these flipping 4-cycles for all $1\le i\le h-1$ are edge-disjoint, and they all avoid~$f$.
\end{proof}

Recall that by Lemma~\ref{lem:short}, short cycles correspond to $[--]$-chains of length $3 \le h \le 2\ell-3$.
For each short cycle represented by such a chain~$C$ of length~$h$, we now specify a chain~$g(C)$ on the same cycle, called a \emph{gluing chain}, and we also specify a parent chain~$p(g(C))$ of~$g(C)$ that belongs to a cycle of a bigger range; that is, a short cycle with a $[--]$-chain of length~$h+2$ or a long cycle.
These definitions are illustrated in Figure~\ref{fig:pgfunc}.
Specifically, for a short cycle represented by a $[--]$-chain $C={*}\,u_1\,{*}\cdots u_{h-1}\,{*}$ of length $3\le h \le 2\ell-3$ we define
\begin{equation}
\label{eq:g}
g(C):=\begin{cases}\ C &\text{if }h<2\ell-3,\\
\ a\,{*}\,u_3\,{*}\cdots u_{h-1}\,{*} &\text{if } h=2\ell-3 \text{ and }u_1=\varepsilon,\\
\ {*}\,u_1\,{*}\cdots u_{h-3}\,{*}\,b &\text{if } h=2\ell-3 \text{ and }u_1\ne\varepsilon,
\end{cases}
\end{equation}
where $a:=0\,u_1\,1\,u_2\in D$ and $b:=u_{h-2}\,0\,u_{h-1}\,1\in D$.
Note that in the second case, $g(C)$ is a $[+-]$-chain of length~$h-2$, and in the third case, $g(C)$ is a $[-+]$-chain of length~$h-2$.
By Lemma~\ref{lem:short}, in each case the chain~$g(C)$ belongs to the same short cycle as $C$.
To define $p(g(C))$, let $0<i<h$ be the smallest index such that $u_i\ne \varepsilon$.
We let $p(g(C))$ be the parent of $g(C)$ obtained by replacing the leftmost matched 01-pair in~$u_i$ by two *s.
That is, given the left factorization $u_i=0\,v\,1\,w$ with $v,w\in D$, then $p(g(C))$ is obtained from~$g(C)$ by replacing $u_i$ with ${*}\,v\,{*}\,w$.
Observe that in all three cases of~\eqref{eq:g}, the chain $p(g(C))$ has the same type as $g(C)$ ($[--]$, $[+-]$, or $[-+]$, respectively).
As a consequence of Lemmas~\ref{lem:short} and~\ref{lem:long}, the chain $p(g(C))$ therefore belongs to a short cycle of bigger range if $h<2\ell-3$ and to a long cycle if $h=2\ell-3$.

\begin{figure}
\includegraphics{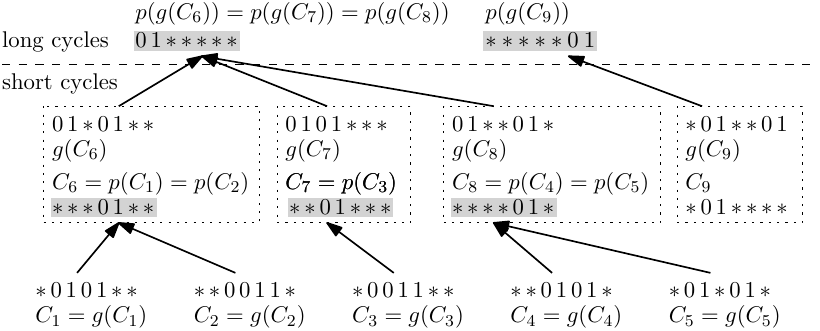}
\caption{For $n=2m+1=7$ and $\ell=4$ we have 9 short cycles represented by the $[--]$-chains $C_1,\dots,C_9$ of length~3 or $5=2\ell-3$.
The dotted boxes indicate chains~$C_i$ with $C_i\neq g(C_i)$, but with~$C_i$ and~$g(C_i)$ on the same short cycle.
Arrows connect $g(C_i)$ with~$p(g(C_i))$.
The top two chains are contained in a long cycle (in this case, in the same cycle).
The roots of the trees mentioned in the proof of Lemma~\ref{lem:4cyc-global} are highlighted in gray.
\label{fig:pgfunc}
}
\end{figure}

We now show that flipping 4-cycles between all gluing chains and their selected parents can be chosen to be pairwise edge-disjoint.

\begin{lemma}
\label{lem:4cyc-global}
There is a set $\cF_{n,\ell}$ of pairwise edge-disjoint flipping 4-cycles, each between~$g(C)$ and~$p(g(C))$ for all short cycles~$C$ of the cycle factor $\cC_{n,\ell}$ defined in~\eqref{eq:cfac}.
\end{lemma}

In this lemma we identify a short cycle~$C$ with its corresponding $[--]$-chain of length $3\le h\le 2\ell-3$.

\begin{proof}
We consider the directed graph on all short chains as nodes and arcs from~$g(C)$ to~$p(g(C))$ for all short cycles~$C$, and in this graph we consider all nontrivial components; see Figure~\ref{fig:pgfunc}.
As the chain~$p(g(C))$ is longer than~$g(C)$, these components are trees that are oriented towards a set of roots, and by the definition of the mappings~$g$ and~$p$, these roots are $[--]$-, $[+-]$- and $[-+]$-chains of length~$2\ell-3$.
We select one flipping 4-cycle into~$\cF_{n,\ell}$ for every arc in each of these trees, and the selection is done by processing each tree independently, starting at the root and descending towards its leaves, along parent-child pairs of chains.
We show that this can be done so that all selected 4-cycles are pairwise edge-disjoint.
Clearly, these 4-cycles can only share an edge on a chain~$g(C)$ or~$p(g(C))$ for some short cycle~$C$.

First, observe that $g(C)$ has length~1 only if $\ell=3$ and $C$ has length~$h=3=2\ell-3$.
In this case the arc from~$g(C)$ to~$p(g(C))$ is an entire tree, so we can take the single flipping 4-cycle between~$g(C)$ and~$p(g(C))$, which exists by Lemma~\ref{lem:4cyc}, into $\cF_{n,\ell}$.

For the rest of the proof we assume that $\ell\geq 4$, so the nodes of all trees are chains of length at least~3, and all interior nodes are chains of length at least~5.
Let $C$ denote the current chain, which is the root of some currently unprocessed subtree, and let $h\ge 5$ be its length.
By Lemma~\ref{lem:4cyc-local}, we may select edge-disjoint flipping 4-cycles to all children of~$C$ in the tree, one for each incoming arc to~$C$, all avoiding the edge~$f$ of~$C$ that was previously chosen for a flipping 4-cycle between~$C$ and~$p(C)$.
We add to the set~$\cF_{n,\ell}$ all those 4-cycles, and we proceed recursively in each subtree below~$C$.
\end{proof}

\section{Flipping 6-cycles}
\label{sec:flip6}

\begin{figure}[b!]
\includegraphics{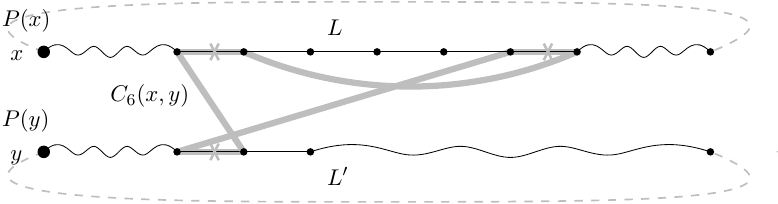}
\caption{
Two long cycles $L,L'$ from the cycle factor $\cC_{n,\ell}$ joined by taking the symmetric difference with a flipping 6-cycle.
The paths~$P(x)$ and~$P(y)$ from the set~$Y_\ell$ (solid black) lying on the two long cycles traverse the 6-cycle~$C_6(x,y)$ (solid gray) as shown.
The symmetric difference yields paths~$P'(x)$ and~$P'(y)$ that have flipped end vertices.
}
\label{fig:c6xy}
\end{figure}

In this section we define certain 6-cycles between the top two levels in~$Q_{n,\ell}$ (where $n=2m+1$, $m\geq 1$, $2\le \ell\le m+1$); that is, between levels~$m+\ell-1$ and~$m+\ell$ in~$Q_n$, that can be used to join pairs of long cycles from the cycle factor~$\cC_{n,\ell}$ defined in~\eqref{eq:cfac}.
Recall that by Lemma~\ref{lem:Ypaths}, in the top two levels the cycle factor~$\cC_{n,\ell}$ consists of paths formed by the edges in~$Y_\ell$, and these paths can be identified by their first vertices $C^+_{2\ell-3,2\ell-3}$.
In the following, whenever referring to paths in~$Y_\ell$ we mean maximal paths, i.e., the components formed by the edges in~$Y_\ell$ (and not proper subpaths of these).

Recall that a flipping 4-cycle for a pair of paths has one edge in common with each of these paths, and taking the symmetric difference of the edge sets of the paths and the 4-cycle yields two paths on the same vertex set with flipped end vertices.
A flipping 6-cycle works very similarly, albeit its intersection pattern with the two paths is more complicated; see Figure~\ref{fig:c6xy}.
Specifically, such a flipping 6-cycle shares \emph{two} non-incident edges with one of the paths, and one edge with the other path.
The precise conditions are stated in Lemma~\ref{lem:6cyc} below.
Let us emphasize that the flipping 6-cycles we use are by definition edge-disjoint with all flipping 4-cycles considered in the previous section, as the 4-cycles are all between levels~$m-\ell+2$ and~$m+\ell-1$ of~$Q_n$, so none of them uses any edges from the top two levels.

We say that two vertices $x,y \in C^+_{2\ell-3,2\ell-3}$ form a \emph{flippable pair} $(x,y)$, if~$0^{2\ell-3}x$ and~$0^{2\ell-3}y$ have the form
\begin{equation}
\label{eq:xy}
\begin{aligned}
  0^{2\ell-3}x &= u_0'\,0\,u_1\cdots 0\,u_d\,0\,u_{d+1}\,0\,1\,1\,v_d\,1\cdots v_2\,1\,v_1\,1, \\
  0^{2\ell-3}y &= u_0'\,0\,u_1\cdots 0\,u_d\,0\,u_{d+1}\,1\,0\,1\,v_d\,1\cdots v_2\,1\,v_1\,1
\end{aligned}
\end{equation}
with~$d\geq 0$ and $u_0',u_1,\ldots,u_{d+1},v_1,\ldots,v_d\in D$.

\begin{wrapfigure}{r}{0.55\textwidth}
\includegraphics{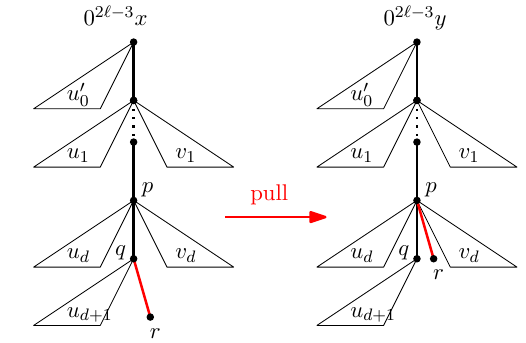}
\caption{The pull operation on rooted trees.}
\label{fig:pull}
\end{wrapfigure}
Recall from Section~\ref{sec:struc} that~$0^{2\ell-3}x$ and~$0^{2\ell-3}y$ can be viewed as rooted trees from~$\cT_{n,\ell}$, where all but one spine edge is contained in the subtree~$u_0'$.
If~$(x,y)$ is a flippable pair, then the tree~$0^{2\ell-3}y$ is obtained from the tree~$0^{2\ell-3}x$ by moving a pending edge from a vertex in the rightmost subtree to its predecessor.
Specifically, the pending edge~$(q,r)$ must form the rightmost subtree of a vertex~$q$ in the rightmost subtree of $0^{2\ell-3}x$, and this edge is removed from $q$ and reattached to the predecessor~$p$ of~$q$ to become the subtree directly right of the edge~$(p,q)$.
We refer to this as a \emph{pull operation}; see Figure~\ref{fig:pull}.

Any 6-cycle between the top two levels of~$Q_{n,\ell}$ can be uniquely encoded as a string of length~$n$ over $\{0,1,*\}$ with $m+\ell-2$ many~1s, $m-\ell$ many~0s and three~$*$s.
The 6-cycle corresponding to this string is obtained by substituting the three $*$s by all six combinations of at least two different symbols from~$\{0,1\}$.
For any flippable pair~$(x,y)$ as in~\eqref{eq:xy}, let $u_0$ be such that $u_0'=0^{2\ell-3}u_0$, and define the \emph{flipping 6-cycle}
\begin{equation}
\label{eq:c6xy}
  C_6(x,y):= u_0\,1\,u_1\cdots 1\,u_d\,{*}\,u_{d+1}\,{*}\,{*}\,1\,0\,v_d\,0\cdots v_2\,0\,v_1.
\end{equation}
Moreover, we let $\cS_{n,\ell}$ denote the set of all those 6-cycles, obtained as the union of~$C_6(x,y)$ over all flippable pairs~$(x,y)$.

The following result was proved in~\cite[Proposition~3]{gregor-muetze-nummenpalo:18}.
For any $x\in C_{2\ell-3,2\ell-3}^+$, we write $P(x)$ for the path from~$Y_\ell$ that starts at the vertex~$x$.

\begin{lemma}
\label{lem:6cyc}
The 6-cycles $C_6(x,y)\in \cS_{n,\ell}$ defined in~\eqref{eq:c6xy} have the following properties:
\begin{enumerate}[label=(\roman*)]
\item
Let~$(x,y)$ be a flippable pair.
The 6-cycle $C_6(x,y)$ intersects~$P(x)$ in two non-incident edges and it intersects~$P(y)$ in a single edge.
Moreover, the symmetric difference of the edge sets of the two paths $P(x)$ and $P(y)$ with the 6-cycle~$C_6(x,y)$ gives two paths~$P'(x)$ and~$P'(y)$ on the same set of vertices as~$P(x)$ and~$P(y)$, connecting~$x$ with the last vertex of~$P(y)$, and $y$ with the last vertex of~$P(x)$, respectively.

\item
For any flippable pairs~$(x,y)$ and~$(x',y')$, the 6-cycles~$C_6(x,y)$ and~$C_6(x',y')$ are edge-disjoint.

\item
For any flippable pairs~$(x,y)$ and~$(x,y')$, the two pairs of edges that the two 6-cycles~$C_6(x,y)$ and~$C_6(x,y')$ have in common with the path~$P(x)$ are not interleaved, but one pair appears before the other pair along the path.
\end{enumerate}
\end{lemma}

\section{Proof of Theorem~\ref{thm:gmlc}}
\label{sec:connect}

With Lemmas~\ref{lem:long},~\ref{lem:4cyc-global}~and~\ref{lem:6cyc} in hand, we are now ready to prove Theorem~\ref{thm:gmlc}.
Recall that $\cT_{n,\ell}=0^{2\ell-3}C^+_{2\ell-3,2\ell-3}$ is the set of rooted trees on $m+\ell-1$ edges with a root of degree at least~2 and the leftmost leaf in depth at least~$2\ell-3$.
Each $x\in \cT_{n,\ell}$ has the form $x=0^{2\ell-3}u_0\,1 \cdots u_{2\ell-4}\,1\,u_{2\ell-3}$ with $u_0,\ldots,u_{2\ell-3}\in D$, $u_{2\ell-3}\neq \varepsilon$.
Considering the right factorization $u_{2\ell-3}=v\,0\, w\,1$ with $v,w\in D$, we say that $x$ is \emph{right-empty} if $w=\varepsilon$, and we say that $x$ is \emph{right-full} if $u_0=\cdots=u_{2\ell-4}=v=\varepsilon$ (and hence $w\neq \varepsilon$).
In terms of trees, $w$ is the subtree rooted at the rightmost child of the root of~$x$.

\begin{proof}[Proof of Theorem~\ref{thm:gmlc}]
The case $\ell=1$ of the theorem was proved in \cite{MR3483129}, so we now assume that $\ell\geq 2$.
Let $m\geq 1$, $n:=2m+1$, $2\le \ell\le m+1$, and consider the subgraph $Q_{n,\ell}$ of the $n$-cube induced by the middle $2\ell$ levels.

Let~$\cC_{n,\ell}$ be the cycle factor in the graph~$Q_{n,\ell}$ defined in~\eqref{eq:cfac}, let $\cF_{n,\ell}$ be the set of flipping 4-cycles from Lemma~\ref{lem:4cyc-global}, and let~$\cS_{n,\ell}$ be the set of flipping 6-cycles defined in~\eqref{eq:c6xy}.
By the choice of~$\cF_{n,\ell}$ in Lemma~\ref{lem:4cyc-global}, the symmetric difference $\cC_{n,\ell}\bigtriangleup \cF_{n,\ell}$ of the edge sets of $\cC_{n,\ell}$ and $\cF_{n,\ell}$ is a cycle factor with the property that every cycle traverses a path from the set~$Y_\ell$ of edges between the top two levels of~$Q_{n,\ell}$ (levels $m+\ell-1$ and $m+\ell$ in $Q_n$), so to complete the proof it is enough to show that long cycles can be joined to a single cycle via flipping 6-cycles from~$\cS_{n,\ell}$.

Consider two long cycles $L,L'$ containing paths $P(x),P(y)\in Y_\ell$ with first vertices $x,y\in C_{2\ell-3,2\ell-3}^+$, respectively, such that $(x,y)$ is a flippable pair.
By Lemma~\ref{lem:6cyc}~(i), $(L\cup L')\bigtriangleup C_6(x,y)$ forms a single cycle on the same vertex set as~$L\cup L'$, i.e., this joining operation reduces the number of long cycles in the factor by one; see Figure~\ref{fig:c6xy}.
Recall that in terms of rooted trees, $0^{2\ell-3}y\in\cT_{n,\ell}$ is obtained from $0^{2\ell-3}x\in\cT_{n,\ell}$ by a pull operation; see Figure~\ref{fig:pull}.
We repeat this joining operation until all long cycles in $\cC_{n,\ell}$ are joined to a single cycle.
For this purpose we define an auxiliary graph~$\cH_{n,\ell}$ whose nodes are the equivalence classes of trees from~$\cT_{n,\ell}$ under heavy and light rotations; see~Figure~\ref{fig:hlrot}.
By Lemma~\ref{lem:long}, the nodes of this graph correspond to maximal sets of paths from~$Y_\ell$ that lie on the same long cycle.
Moreover, for the edge set of $\cH_{n,\ell}$ we take all pairs of sets that contain the two trees that form a flippable pair (differing by a pull operation).

To complete the proof of Theorem~\ref{thm:gmlc}, it therefore suffices to prove that the auxiliary graph~$\cH_{n,\ell}$ is connected.
Indeed, if~$\cH_{n,\ell}$ is connected, then we can pick a spanning tree in~$\cH_{n,\ell}$ corresponding to a collection of flipping 6-cycles $\cS_{n,\ell}'\seq \cS_{n,\ell}$, such that the symmetric difference $\cC_{n,\ell} \bigtriangleup (\cF_{n,\ell} \cup \cS_{n,\ell}')$ forms a Hamilton cycle in~$Q_{n,\ell}$.
Of crucial importance here are properties~(ii) and~(iii) in Lemma~\ref{lem:6cyc}, which ensure that whatever subset of flipping 6-cycles we use in this joining process, they will not interfere with each other, guaranteeing that each flipping 6-cycle indeed reduces the number of long cycles by one, as desired.
Recall also that flipping 4-cycles and flipping 6-cycles are edge-disjoint by definition, so they do not interfere with each other either.

\begin{figure}
\makebox[0cm]{ 
\includegraphics{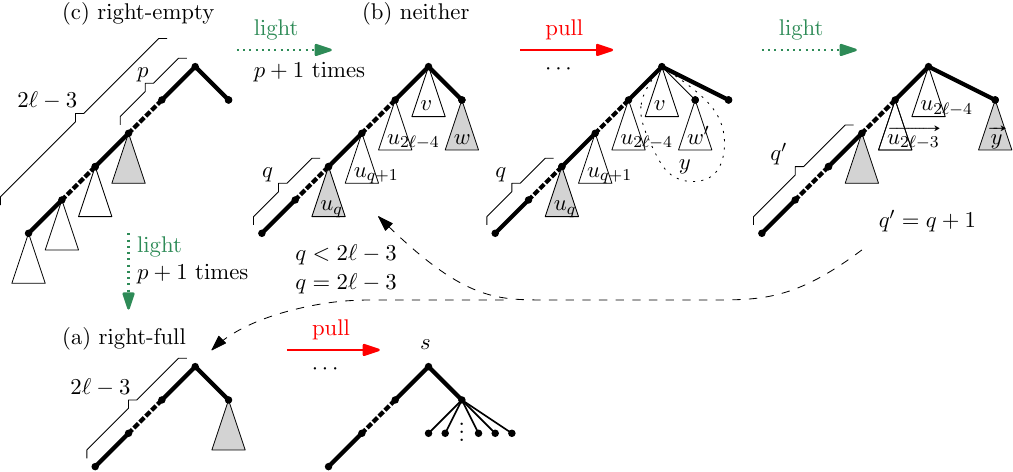}
}
\caption{
Transformation of trees from~$\cT_{n,\ell}$ into the tree~$s$ in the proof of Theorem~\ref{thm:gmlc}.
Gray subtrees are nonempty, white subtrees are possibly empty.
Spine edges are drawn bold.
}
\label{fig:connect}
\end{figure}

At this point we reduced the problem of proving that $Q_{n,\ell}$ has a Hamilton cycle to showing that the auxiliary graph~$\cH_{n,\ell}$ is connected, which is much easier.
Indeed, all we need to show is that any rooted tree from~$\cT_{n,\ell}$ can be transformed into any other tree from~$\cT_{n,\ell}$ by a sequence of heavy rotations, light rotations, pulls and their inverse operations (actually, we shall only use light rotations and pulls in our proof).
Recall that heavy and light rotations correspond to following the same long cycle, and a pull corresponds to a joining operation.
For this we show that any rooted tree $x\in \cT_{n,\ell}$ can be transformed into the special tree $s:=0^{2\ell-3}\,1^{2\ell-3}\,0\,0\,1 \cdots 0\,1\,0\in \cT_{n,\ell}$, i.e., a right-full tree with a star rooted at the center as the rightmost child of the root.
To achieve this, we distinguish three cases; see Figure~\ref{fig:connect}.
\begin{enumerate}
\item[(a)] \emph{$x$ is right-full.}
We pull edges within the subtree rooted at the rightmost child of the root until this subtree is a star rooted at the center, which produces the desired tree~$s$.

\item[(b)] \emph{$x$ is neither right-full nor right-empty.}
Let $q$ be the distance between the left end vertex of the spine and the closest nonempty subtree attached to the spine.
That is, $q$ is the smallest index~$i$ such that $u_i\ne \varepsilon$ if such $u_i$ exists, and $q=2\ell-3$ if $u_0=\cdots=u_{2\ell-4}=\varepsilon$ (as $x$ is not right-full, $v\neq \varepsilon$ in this case).
First, we repeatedly pull the single edge from the rightmost leaf of~$x$ all the way towards the root, arriving at a right-empty tree.
After a single light rotation we obtain a tree that is not right-empty and that is either right-full (if $q=2\ell-3$), and then we conclude the argument as in case~(a), or in which the distance~$q'$ between the left end vertex of the spine and the closest nonempty subtree has increased by~1 (if $q<2\ell-3$), i.e., we have $q'=q+1$.
Repeating the argument from case~(b) hence terminates after at most $2\ell-2$ repetitions.

\item[(c)] \emph{$x$ is right-empty.}
Let $p$ be the distance between the root and the closest nonempty subtree attached to the spine.
After exactly $p+1$ light rotations, we are in case~(a) or~(b).
\end{enumerate}

This shows that~$\cH_{n,\ell}$ is connected, and thus completes the proof.
\end{proof}

In our construction of the Hamilton cycle $\cC_{n,\ell} \bigtriangleup (\cF_{n,\ell} \cup \cS_{n,\ell}')$ described in the previous proof, the choice of the set of flipping 6-cycles~$\cS_{n,\ell}'$ depends on the global structure of long cycles.
As explained in Section~\ref{sec:long}, our Lemma~\ref{lem:long} does not provide any information about the global structure of these cycles, which is the main obstacle that prevents us from translating our construction to a polynomial-time algorithm for computing the next vertex on the Hamilton cycle, i.e., to an algorithm that only has local information about the current vertex in deciding which bit to flip next.
Without explicitly constructing all long cycles, which may take time and space that are exponential in~$n$, we do not even know how many 6-cycles should be selected into the set~$\cS_{n,\ell}'$.

\section{Proof of Theorem~\ref{thm:gk}}
\label{sec:gk}

For any chain~$C$, we let $|C|$ denotes its length, i.e., the number of *s in~$C$.
For any chain~$C$ with $|C|\geq 2$, we let $f(C)$ and $\ell(C)$, respectively, denote the chains obtained by replacing the first two *s or the last two *s in~$C$ by~0 and~1.
Note that if $|C|\geq 2$, then we have $f(\ell({*}C{*}))=\ell(f({*}C{*}))$.
For any chain~$C$, we denote its bottom end vertex by~$b(C)$ and its top end vertex by~$t(C)$.
Recall that $b(C)$ and $t(C)$ are obtained from~$C$ by replacing all *s by~0s or~1s, respectively.

Our goal is to order the chains of the Greene-Kleitman SCD in $Q_n$, $n\geq 2$, so that any consecutive pair of chains can be joined, alternatingly at their top ends or bottom ends, to a Hamilton cycle.
Formally, given an SCD in~$Q_n$, a \emph{cycle ordering} is a sequence $C_1,\ldots,C_k$ of the chains of this SCD for which there exists a set of edges~$E$ of~$Q_n$ such that $C_1\cup C_2\cup\cdots\cup C_k\cup E$ is a Hamilton cycle that traverses the chains in this order, either from bottom to top or vice versa, and moreover if $|C_i|=0$ for some $i\in\{1,\ldots,k\}$ then the two edges from~$E$ incident with this one-vertex chain must have their other end vertices in distinct levels.
The following simple but powerful lemma shows that the direction in which each chain is traversed along the Hamilton cycle (upwards or downwards) is determined only by the chain length.

\begin{lemma}
\label{lem:dir}
Let $\Lambda_n$ be a cycle ordering of chains of an SCD in~$Q_n$, $n\geq 2$.
Then in the corresponding Hamilton cycle, any two chains $C$ and $C'$ with $|C|\equiv |C'| \cmod 4$ are traversed in the same direction.
\end{lemma}

This lemma holds for arbitrary SCDs, in particular for the SCD arising from the Greene-Kleitman construction.

\begin{proof}
This is a consequence of the following observations:
After traversing a chain~$C$ with $|C|\geq 2$, the next chain in the cycle ordering has either length~$|C|-2$ or $|C|+2$, and is traversed in the opposite direction in both cases.
Similarly, if $|C|=1$, the next chain in the cycle ordering has either length~$|C|=1$ or~$|C|+2=3$, and is traversed in the same direction or the opposite direction, respectively.
Moreover, if $|C|=0$, then the previous and next chain have length~$|C|+2=2$, and by the condition that the two edges of the cycle incident with this one-vertex chain must have their other end vertices in distinct levels, both chains of length~2 are traversed in the same direction.
\end{proof}

We now define a cycle ordering~$\Lambda_n$, $n\geq 2$, for the Greene-Kleitman SCD; see Figure~\ref{fig:scd} for illustration.
The corresponding Hamilton cycle is oriented so that it traverses the longest chain ${*}^n$, which will be the first in the ordering~$\Lambda_n$, from bottom to top.
Our construction works inductively, and the induction step goes from~$n$ to $n+2$, with separate rules for even and odd~$n$.
The base cases are $n=0$ and $n=1$, for which the entire cube consists only of a single vertex and a single edge, respectively, so for these cases the notion of a cycle ordering is not defined.

For even~$n$, we define $\Lambda_0:=\varepsilon$, and for $n\geq 0$ and given $\Lambda_n=:C_1,\ldots,C_N$ we define $\Lambda_{n+2}:=\rho(\Lambda_n)=\rho(C_1),\ldots,\rho(C_N)$ with
\begin{subequations}
\label{eq:lambda}
\begin{equation}
\label{eq:rev}
\rho(C):=\begin{cases}
\lambda(C) & \text{if } |C|\equiv n \cmod 4, \\
\lambda(C)^R & \text{if } |C|\not\equiv n \cmod 4,
\end{cases}
\end{equation}
and
\begin{equation}
\label{eq:desc0}
\lambda(C):=\begin{cases}
{*}C{*}, f({*}C{*}), f(\ell({*}C{*})), \ell({*}C{*}) & \text{if } |C|\geq 2, \\
{*}C{*}, 0C1 & \text{if } |C|=0.
\end{cases}
\end{equation}
We call the chains of $\lambda(C)$ arising from~$C$ the \emph{descendants of~$C$}.
Essentially this rule replaces each chain~$C$ in $\Lambda_n$ by its descendants $\lambda(C)$, where the order of descendants can be reversed, indicated by the superscript~$R$, depending on the length of~$C$ modulo~4.

For odd~$n$, we define $\Lambda_1:={*}$, and for $n\geq 1$ and given $\Lambda_n$ we define $\Lambda_{n+2}:=\rho(\Lambda_n)$, where $\rho$ is as before and
\begin{equation}
\label{eq:desc1}
\lambda(C):=\begin{cases}
{*}C{*}, \ell({*}C{*}), \ell(f({*}C{*})), f({*}C{*}) & \text{if } |C|\geq 3, \\
{*}C{*}, \ell({*}C{*}), f({*}C{*}) & \text{if } |C|=1.
\end{cases}
\end{equation}
\end{subequations}

\begin{figure}
\includegraphics{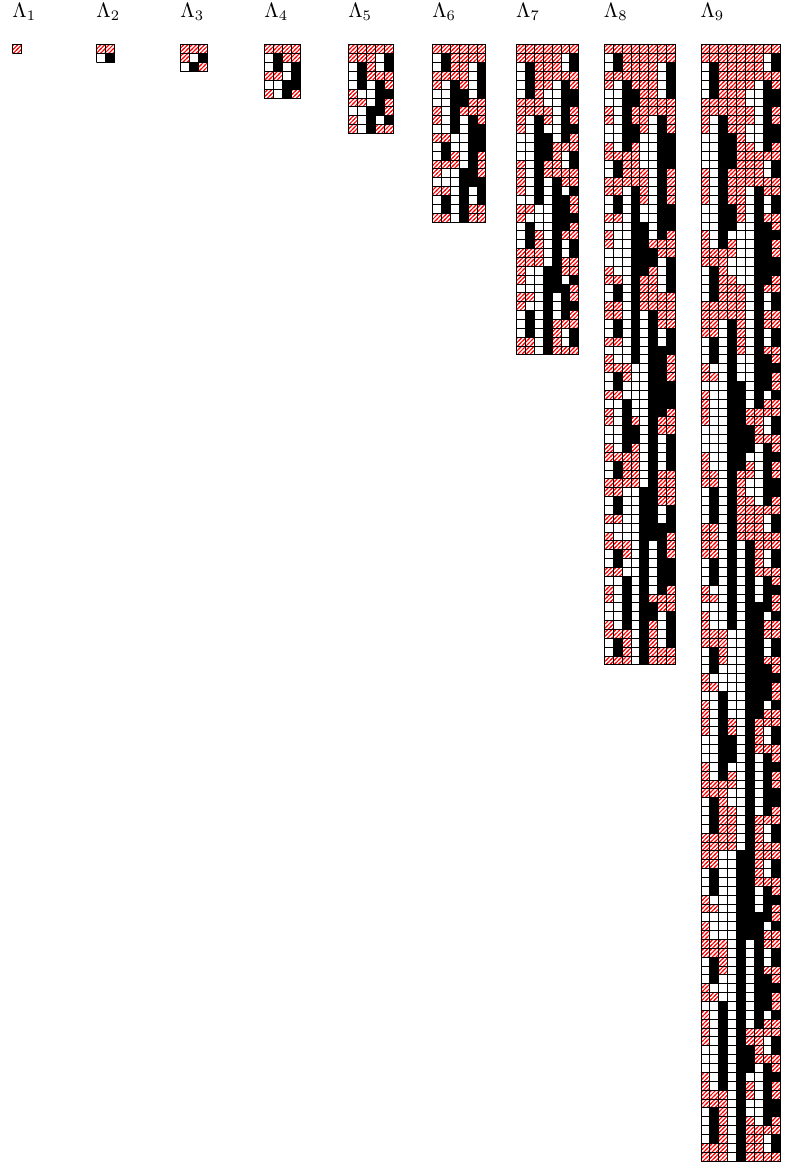}
\caption{The cycle orderings $\Lambda_n$ of Greene-Kleitman chains for $n=1,\ldots,9$.
In this figure, 1-bits are drawn as black squares, 0-bits as white squares, and *s as hatched squares.
Figure~\ref{fig:hc7}~(f) is obtained by traversing the chains of~$\Lambda_7$ up and down alternatingly.
}
\label{fig:scd}
\end{figure}

\begin{lemma}
\label{lem:lambda-scd}
$\Lambda_n$ contains every chain of the Greene-Kleitman SCD exactly once.
\end{lemma}

\begin{proof}
Observe that for even~$n$, if $|C|\geq 2$ in~\eqref{eq:desc0}, then $*C*$ is a $[--]$-chain of length $|C|+2$, $f(*C*)$ is a $[+-]$-chain of length~$|C|$, $\ell(*C*)$ is a $[-+]$-chain of length~$|C|$, and $f(\ell(*C*))$ is a $[++]$-chain of length~$|C|-2$, and if $|C|=2$, then the lattice path $f(\ell(*C*))\in D$ touches the line $y=0$ at least three times.
If $|C|=0$, on the other hand, then $*C*$ is a $[--]$-chain of length~2, $0C1$ is a $[++]$-chain of length~1, and the lattice path $0C1\in D$ touches the line $y=0$ exactly twice.
Overall, the inductive rule~\eqref{eq:desc0} produces only distinct chains, and every chain of the Greene-Kleitman SCD in~$Q_{n+2}$ is produced.
A similar argument works for odd~$n$, showing that $\Lambda_n$ indeed contains every chain of the Greene-Kleitman SCD exactly once.
\end{proof}

To complete the proof of Theorem~\ref{thm:gk}, it remains to show that any two consecutive chains in $\Lambda_n$ can be joined by an edge between their top ends or bottom ends alternatingly.
For this we need the following simple lemmas that guarantee these connecting edges.

\begin{lemma}
\label{lem:conn-desc}
For any $n\geq 2$ and any chain~$C$ with $|C|\geq 2$, the chains~$C$ and~$f(C)$, and the chains~$C$ and $\ell(C)$ are connected both at their top and bottom ends in~$Q_n$.
\end{lemma}

All the connecting edges between top and bottom ends among the descendants of a chain guaranteed by Lemma~\ref{lem:conn-desc} are shown in Figure~\ref{fig:desc}.

\begin{figure}
\centering
\includegraphics{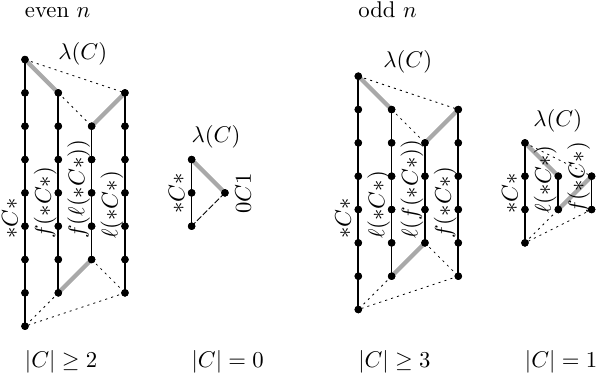}
\caption{Connections between top and bottom ends of the descendants $\lambda(C)$ of a chain~$C$, as guaranteed by Lemmas~\ref{lem:conn-desc} and~\ref{lem:conn-1}.
Bold gray edges are the connections used along the Hamilton cycle.
Dotted edges are present but not used.}
\label{fig:desc}
\end{figure}

\begin{proof}
By symmetry, it suffices to prove the lemma for the chains~$C$ and~$f(C)$.
As $|C|\geq 2$, we may consider the first two *s in~$C$ and write $C=u\,{*}\,v\,{*}\,C'$ with $u,v\in D$, which yields $f(C)=u\,0\,v\,1\,C'$.
From this we obtain that $b(C)=u\,0\,v\,0\,b(C')$ and $b(f(C))=u\,0\,v\,1\,b(C')$, so $b(C)$ and $b(f(C))$ differ exactly in the bit after the valley~$v$.
Similarly, we have $t(C)=u\,1\,v\,1\,t(C')$ and $t(f(C))=u\,0\,v\,1\,t(C')$, so $t(C)$ and $t(f(C))$ differ exactly in the bit after the valley~$u$.
\end{proof}

The next two lemmas are illustrated in the top part of Figure~\ref{fig:join}.

\begin{lemma}
\label{lem:conn-bot}
For any $n\geq 2$ and any two chains~$C,C'$ connected at their bottom ends in~$Q_n$, we have that ${*}C{*}$ and ${*}C'{*}$ are connected at their bottom ends in~$Q_{n+2}$.
\end{lemma}

\begin{proof}
Note that $b({*}C{*})=0\,b(C)\,0$ and $b({*}C'{*})=0\,b(C')\,0$.
Consequently, as~$C$ and~$C'$ are connected at their bottom ends, $b(C)$ and $b(C')$ differ in exactly one bit, implying that $b(*C*)$ and $b(*C'*)$ also differ in exactly one bit.
\end{proof}

\begin{lemma}
\label{lem:conn-top0}
For any $n\geq 2$ and any chain~$C$ with $|C|\geq 2$ in~$Q_n$, we have that $\ell({*}C{*})$ and $\ell({*}f(C){*})$ are connected at their bottom ends in~$Q_{n+2}$.
Specifically, if $|C|=2$ and $C = u\,{*}\,v\,{*}\,w$ with $u,v,w\in D$, then we have
\begin{subequations}
\begin{equation}
\label{eq:diff03}
\begin{aligned}
\ell({*}C{*})    &= {*}\,u\,{*}\,v\,0\,w\,1, \\
\ell({*}f(C){*}) &= 0\,u\,0\,v\,1\,w\,1.
\end{aligned}
\end{equation}
If $|C|\geq 3$ and $C = C'\,{*}\,u$ with $u\in D$ and $|C'|\geq 2$, then we have
\begin{equation}
\label{eq:diff02}
\begin{aligned}
\ell({*}C{*})    &= {*}\quad C'\,\,\,0\,u\,1, \\
\ell({*}f(C){*}) &= {*}f(C')\,0\,u\,1.
\end{aligned}
\end{equation}
\end{subequations}
Consequently, in the first case, the chains $\ell({*}C{*})$ and $\ell({*}f(C){*})$ differ in exactly three positions, and in the second case, they differ in exactly two positions.
\end{lemma}

\begin{proof}
First of all, the relations \eqref{eq:diff03} and \eqref{eq:diff02} can be verified directly using the definition of~$f$ and~$\ell$.
From \eqref{eq:diff03} we obtain that $b(\ell({*}C{*}))$ and $b(\ell({*}f(C){*}))$ differ exactly in the bit after~$v$.
From \eqref{eq:diff02} we obtain that $b(\ell({*}C{*}))=0\,b(C')\,0\,u\,1$ and $b(\ell({*}f(C){*}))=0\,b(f(C'))\,0\,u\,1$, and as the substrings~$b(C')$ and~$b(f(C'))$ differ in exactly one bit by Lemma~\ref{lem:conn-desc}, this also holds for the entire strings.
\end{proof}

Note that if $|C|\geq 3$ (not if $|C|=2$), then by~\eqref{eq:diff02} the two chains mentioned in Lemma~\ref{lem:conn-top0} are also connected at their top ends, but this connection is irrelevant for us.

\begin{figure}
\makebox[0cm]{ 
\includegraphics{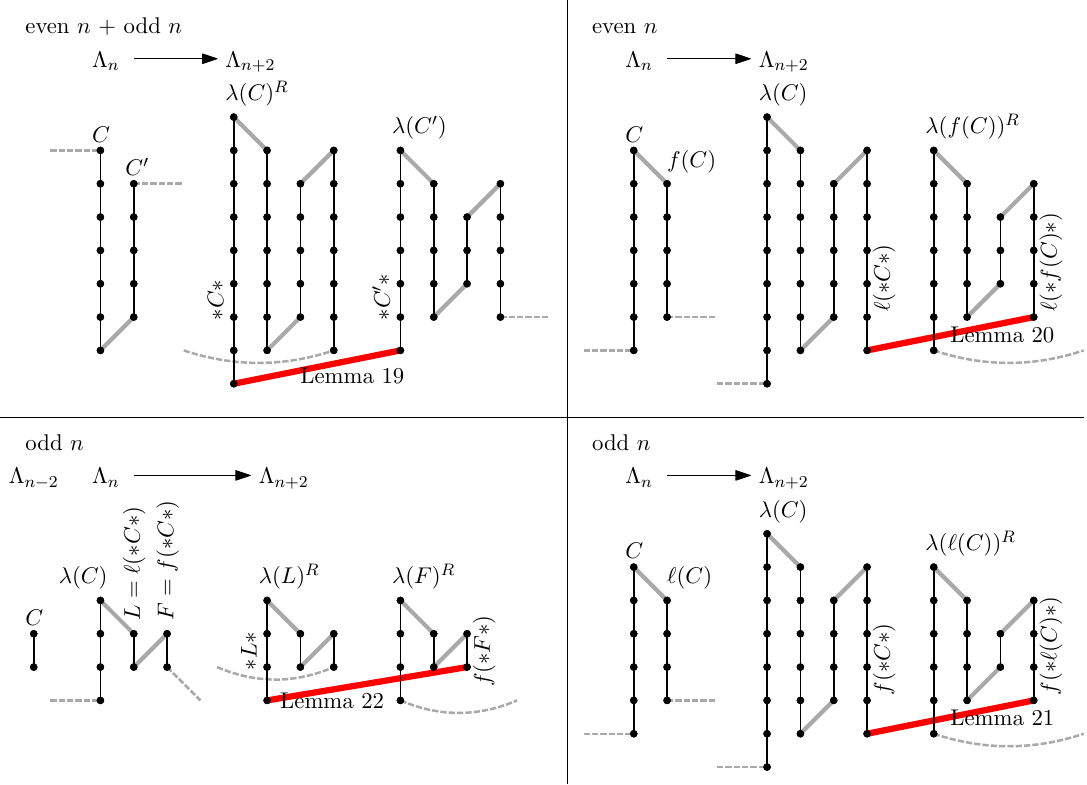}
}
\caption{Joining of the descendants of two consecutive chain from~$\Lambda_n$ in the induction step $n\rightarrow n+2$, via the thick edges guaranteed by Lemmas~\ref{lem:conn-bot}--\ref{lem:conn-1}.
The dashed edges are connections to preceding and subsequent chains on the Hamilton cycle.}
\label{fig:join}
\end{figure}

\begin{proof}[Proof of Theorem~\ref{thm:gk} (even $n$)]
We show that $\Lambda_n$, $n\geq 2$ even, defined in~\eqref{eq:desc0} is a cycle ordering of the Greene-Kleitman chains, by proving that any consecutive pair of chains is connected at their top or bottom ends alternatingly, starting with the first chain $*^n$ of length~$n$ that is traversed from bottom to top.
We will also establish the following additional property~P: For any two consecutive chains~$C$ and~$C'$ connected at their top ends, we either have $C=f(C')$ or $f(C)=C'$.
These invariants can easily be checked for the induction base case~$n=2$, which is given by $\Lambda_2=**,01$.

For the induction step consider $n\geq 2$ to be even, and assume that $\Lambda_n$ is a cycle ordering satisfying property~P.
By Lemma~\ref{lem:conn-desc}, the descendants $\lambda(C)$ for any chain~$C$ from~$\Lambda_n$ can be joined as shown on the left hand side of Figure~\ref{fig:desc}, so we only need to check the connections between the first and last chains among consecutive groups of descendants.
Indeed, if $C$ and~$C'$ are consecutive in~$\Lambda_n$ and joined at their bottom ends, then $C$ is traversed from top to bottom and $C'$ from bottom to top in the Hamilton cycle; see the top left part of Figure~\ref{fig:join}.
Consequently, by Lemma~\ref{lem:dir}, we have $|C|\not\equiv |{*}^n|=n \cmod 4$ and $|C'|\equiv n \cmod 4$, i.e., by~\eqref{eq:rev} the sequence $\Lambda_{n+2}$ contains $\lambda(C)^R$ and $\lambda(C')$, and indeed, the bottom vertex of the last chain of $\lambda(C)^R$, namely $*C*$, is connected to the bottom vertex of the first chain of~$\lambda(C')$, namely $*C'*$, by Lemma~\ref{lem:conn-bot}.

Similarly, if $C$ and~$C'$ are consecutive in~$\Lambda_n$ and joined at their top ends, then $C$ is traversed from bottom to top and $C'$ from top to bottom in the Hamilton cycle; see the top right part of Figure~\ref{fig:join}.
Consequently, by Lemma~\ref{lem:dir}, we have $|C|\equiv n \cmod 4$ and $|C'|\not\equiv n \cmod 4$, i.e., by~\eqref{eq:rev} the sequence $\Lambda_{n+2}$ contains $\lambda(C)$ and $\lambda(C')^R$, and indeed, the bottom vertex of the last chain of $\lambda(C)$, namely $\ell(*C*)$, is connected to the bottom vertex of the first chain of~$\lambda(C')^R$, namely $\ell(*C'*)$, using that by property~P we have either $C=f(C')$ or $f(C)=C'$, so we can invoke Lemma~\ref{lem:conn-top0}.
Moreover, property~P still holds for~$\Lambda_{n+2}$ by the definition~\eqref{eq:desc0} (note that if $|C|=0$, then we have $0C1=f({*}C{*})$).
\end{proof}

To prove Theorem~\ref{thm:gk} for odd~$n$, we need two additional lemmas, illustrated at the bottom part of Figure~\ref{fig:join}.
Lemma~\ref{lem:conn-top1} is the `dual' version of Lemma~\ref{lem:conn-top0}, with $f$ replaced by $\ell$ and vice versa, which we need as the definitions~\eqref{eq:desc0} and~\eqref{eq:desc1} in the first case differ in the ordering of descendants.
Lemma~\ref{lem:conn-1} deals with the special case of consecutive chains of length~1, a situation that never occurs for even~$n$.

\begin{lemma}
\label{lem:conn-top1}
For any $n\geq 3$ and any chain~$C$ with $|C|\geq 3$ in~$Q_n$, we have that $f({*}C{*})$ and $f({*}\ell(C){*})$ are connected at their bottom ends in~$Q_{n+2}$.
Specifically, if $C=u\,{*}\,C'$ with $u\in D$ and $|C'|\geq 2$, then we have
\begin{equation}
\label{eq:diff12}
\begin{aligned}
f({*}C{*})       &= 0\,u\,1\quad C'\,\,\, {*}, \\
f({*}\ell(C){*}) &= 0\,u\,1\,\ell(C')\,{*},
\end{aligned}
\end{equation}
i.e., the chains $f({*}C{*})$ and $f({*}\ell(C){*})$ differ in exactly two positions.
\end{lemma}

\begin{proof}
The relation \eqref{eq:diff12} can be verified directly using the definition of~$f$ and~$\ell$.
From \eqref{eq:diff12} we obtain that $b(f({*}C{*}))=0\,u\,1\,b(C')\,0$ and $b(f({*}\ell(C){*}))=0\,u\,1\,b(\ell(C'))\,0$, and as the substrings~$b(C')$ and~$b(\ell(C'))$ differ in exactly one bit by Lemma~\ref{lem:conn-desc}, this also holds for the entire strings.
\end{proof}

By~\eqref{eq:diff12}, the two chains mentioned in Lemma~\ref{lem:conn-top1} are also connected at their top ends, but this connection is irrelevant for us.

\begin{lemma}
\label{lem:conn-1}
For any odd~$n\geq 3$ and any chain~$C$ with $|C|=1$ in~$Q_{n-2}$, we have that the bottom end of $L:=\ell({*}C{*})$ is joined to the top end of $F:=f({*}C{*})$ in~$Q_n$, and moreover ${*}L{*}$ and $f({*}F{*})$ are connected at their bottom ends in~$Q_{n+2}$.
Specifically, if $C=u\,{*}\,v$ with $u,v\in D$, then we have
\begin{equation}
\label{eq:diff13}
\begin{aligned}
L &= {*}\,u\,0\,v\,1, & \quad\quad {*}L{*}\ &= {*}\,{*}\,u\,0\,v\,1\,{*}, \\
F &= 0\,u\,1\,v\,{*}, & \quad\quad f({*}F{*}) &= 0\,0\,u\,1\,v\,1\,{*},
\end{aligned}
\end{equation}
i.e., both~$L$ and~$F$, as well as~${*}L{*}$ and~$f({*}F{*})$ differ in exactly three positions.
\end{lemma}

\begin{proof}
The relation \eqref{eq:diff13} can be verified directly using the definition of~$f$ and~$\ell$.
From \eqref{eq:diff13} we obtain that $b(L)=0\,u\,0\,v\,1$ and $t(F)=0\,u\,1\,v\,1$, so these two strings differ exactly in the bit after~$u$.
Similarly, we obtain that $b({*}L{*})=0\,0\,u\,0\,v\,1\,0$ and $b(f({*}F{*}))=0\,0\,u\,1\,v\,1\,0$, so these two strings also differ exactly in the bit after~$u$.
\end{proof}

\begin{proof}[Proof of Theorem~\ref{thm:gk} (odd $n$)]
We show that $\Lambda_n$, $n\geq 3$ odd, defined in~\eqref{eq:desc1} is a cycle ordering of the Greene-Kleitman chains, by proving that any consecutive pair of chains is connected at their top or bottom ends alternatingly, starting with the first chain $*^n$ of length~$n$ that is traversed from bottom to top.
We will also establish the following additional properties~P' and~Q: Property~P' asserts that for any two consecutive chains~$C$ and~$C'$ connected at their top ends, we either have $C=\ell(C')$ or $\ell(C)=C'$.
Property~Q asserts that for any two consecutive chains~$B$ and~$B'$ of length~1 in~$\Lambda_n$, there is a chain~$C$ of length~1 in $\Lambda_{n-2}$ such that $\{B,B'\}=\{\ell(*C*),f(*C*)\}$, i.e., both~$B$ and~$B'$ are descendants of~$C$ (recall the second case of~\eqref{eq:desc1}).
Property~Q implies in particular that no more than two chains of length~1 appear consecutively in~$\Lambda_n$.
These invariants can easily be checked for the induction base case~$n=3$, which is given by $\Lambda_3={*}{*}{*},{*}01,01{*}$.

For the induction step consider $n\geq 3$ to be odd, and assume that $\Lambda_n$ is a cycle ordering satisfying properties~P' and~Q.
By Lemmas~\ref{lem:conn-desc} and~\ref{lem:conn-1}, the descendants $\lambda(C)$ for any chain~$C$ from~$\Lambda_n$ can be joined as shown on the right hand side of Figure~\ref{fig:desc}, so we only need to check the connections between the first and last chains among consecutive groups of descendants.
Indeed, if $C$ and~$C'$ are consecutive in~$\Lambda_n$ and joined at their bottom ends, then $C$ is traversed from top to bottom and $C'$ from bottom to top in the Hamilton cycle; see the top left part of Figure~\ref{fig:join}.
Consequently, by Lemma~\ref{lem:dir}, we have $|C|\not\equiv |{*}^n|=n \cmod 4$ and $|C'|\equiv n \cmod 4$, i.e., by~\eqref{eq:rev} the sequence $\Lambda_{n+2}$ contains $\lambda(C)^R$ and $\lambda(C')$, and indeed, the bottom vertex of the last chain of $\lambda(C)^R$, namely $*C*$, is connected to the bottom vertex of the first chain of~$\lambda(C')$, namely $*C'*$, by Lemma~\ref{lem:conn-bot}.

Similarly, if $C$ and~$C'$ are consecutive in~$\Lambda_n$ and joined at their top ends, then $C$ is traversed from bottom to top and $C'$ from top to bottom in the Hamilton cycle; see the bottom right part of Figure~\ref{fig:join}.
Consequently, by Lemma~\ref{lem:dir}, we have $|C|\equiv n \cmod 4$ and $|C'|\not\equiv n \cmod 4$, i.e., by~\eqref{eq:rev} the sequence $\Lambda_{n+2}$ contains $\lambda(C)$ and $\lambda(C')^R$, and indeed, the bottom vertex of the last chain of $\lambda(C)$, namely $f(*C*)$, is connected to the bottom vertex of the first chain of~$\lambda(C')^R$, namely $f(*C'*)$, using that by property~P' we have either $C=\ell(C')$ or $\ell(C)=C'$, so we can invoke Lemma~\ref{lem:conn-top1}.

The last case to consider are two consecutive chains~$B$ and~$B'$ of length~1 in~$\Lambda_n$, where the bottom end of~$B$ is joined to the top end of~$B'$ if $n\equiv 3 \cmod 4$ or the other way round if $n\equiv 1 \cmod 4$ (recall Lemma~\ref{lem:dir}); see the bottom left part of Figure~\ref{fig:join}.
We only consider the case $n\equiv 3 \cmod 4$, as the other case is symmetric.
By property~Q, we know that $B=\ell(*C*)=:L$ and $B'=f(*C*)=:F$ for some chain~$C$ of length~1 in~$\Lambda_{n-2}$.
By~\eqref{eq:rev}, the sequence $\Lambda_{n+2}$ contains~$\lambda(B)^R$ and $\lambda(B')^R$, and indeed, the bottom vertex of the last chain of~$\lambda(B)^R=\lambda(L)^R$, namely ${*}L{*}$, is connected to the bottom vertex of the first chain of~$\lambda(B')^R=\lambda(F)^R$, namely $f({*}F{*})$, by Lemma~\ref{lem:conn-1}.
Moreover, properties~P' and~Q still hold for~$\Lambda_{n+2}$ by the definition~\eqref{eq:desc1}.
\end{proof}

\subsection{Loopless algorithm}

The following is an immediate consequence of the lemmas we established.

\begin{theorem}
\label{thm:3gray}
For any $n\geq 1$, the ordering of chains~$\Lambda_n$ defined in~\eqref{eq:lambda} is a 3-Gray code, i.e., any two consecutive chains, viewed as strings over the alphabet $\{0,1,*\}$, differ in at most three positions.
\end{theorem}

\begin{proof}
Note that by the definitions~\eqref{eq:desc0} and~\eqref{eq:desc1}, any two consecutive chains in $\lambda(C)$ differ in exactly two positions.
From this the theorem follows by induction on $n$, using Lemmas~\ref{lem:conn-bot}--\ref{lem:conn-1}.
\end{proof}

We now describe an algorithm which for a given dimension~$n$ computes the sequence of chains in the cycle ordering~$\Lambda_n$ defined in~\eqref{eq:lambda}.
Each chain is represented as a string of length~$n$ over the alphabet~$\{0,1,*\}$ (recall Figure~\ref{fig:scd}), and whenever the next chain has been produced by the algorithm, it is visited in the [Visit] step.
By Theorem~\ref{thm:3gray}, at most three entries of this string change between two visits.
One can easily turn this algorithm into a loopless algorithm for computing the entire Hamilton cycle, flipping a single bit in each step, by replacing the [Visit] step by a loop that moves up and down the vertices on the current chain alternatingly.

For even~$n$, a loopless implementation of this Gray code is given as Algorithm~C, i.e., the algorithm requires only~$\cO(1)$ time between any two consecutive [Visit] steps.
The corresponding variant of the algorithm for odd~$n$ is very similar and can be found in the Appendix.
An implementation of both the even and odd case in C++ is available for download and for demonstration on the Combinatorial Object Server~\cite{cos_chains}.
The initialization required by this algorithm is~$\cO(n)$, and the used space is~$\cO(n)$.
Table~\ref{tab:algoC} shows the execution of this algorithm for the case $n=6$.

\begin{algo}{C}{Loopless Gray code for Greene-Kleitman chains in $Q_n$ for even $n=2m$, $m\geq 1$}
The input of the algorithm is an integer~$m\geq 1$, and the produced chains in the cycle ordering~$\Lambda_{2m}$ are visited in step~C2, where the current chain is stored in the array $C=c_1\ldots c_n$ with $c_i\in\{0,1,{*}\}$.
For maintaining the chain efficiently, the algorithm uses auxiliary arrays $p_1\ldots p_n$, $s_0\ldots s_n$, and $t_1\ldots t_{n+1}$, where if $c_i\in\{0,1\}$, then $p_i$ gives the position of the~0 or~1 matched to~$c_i$, and if $c_i={*}$, then $s_i$ is the position of the closest~$*$ to the right of~$c_i$ and $t_i$ is the position of the closest~$*$ to the left of~$c_i$.
The algorithm also maintains several additional arrays to track the recursive structure of~$\Lambda_n$.
Specifically, the chain constructed in dimension~$2i$, $i=1,\ldots,m$, consists of the middle $2i$ entries of~$C$.
The array $d_0\ldots d_m$ is used to determine the dimension~$2d_m$ in which the next construction step is performed.
This array simulates a stack that generates the transition sequence of the Gray code, with $d_m$ being the value on the top of the stack, an idea first used by Bitner, Ehrlich, and Reingold~\cite{MR0424386}.
In addition, the algorithm maintains arrays $l_1\ldots l_m$, $b_1\ldots b_m$, $o_1\ldots o_m$, $q_1\ldots q_m$, where $l_i$ is the length of the chain in dimension~$2i$, $b_i\in\{+,-\}$ indicates whether we are currently in the first case of~\eqref{eq:desc0} ($+$) or the second case ($-$), $o_i\in\{+,-\}$ indicates whether we are in the first case of~\eqref{eq:rev} ($+$) or the second case ($-$), and $q_i\in\{f,\ell,f^{-1},\ell^{-1}\}$ specifies the next operation on the chain in dimension~$2i$.
\begin{enumerate}[label={\bfseries C\arabic*.}, leftmargin=9mm, noitemsep, topsep=3pt plus 3pt]
\item{} [Initialize] Set $c_i\leftarrow {*}$ for $1\le i\le n$, $s_i\leftarrow i+1$ for $0\le i\le n$, and $t_i\leftarrow i-1$ for $1\le i\le n+1$.
Also set $l_i\leftarrow 2i$ for $1\le i\le m$, $b_1\leftarrow -$ and $b_i\leftarrow +$ for $2\le i\le m$, $o_i\leftarrow +$ and $q_i\leftarrow f$ for $1\le i\le m$, $d_i\leftarrow i$ for $0\le i\le m$.
\item{} [Visit] Visit the chain $c_1\ldots c_n$.
\item{} [Select dimension] Set $i\leftarrow d_m$. Terminate if $i=0$.
\item{} [Perform operation] Depending on the value of $q_i$, branch to one of the following four cases. \\
\textbf{C41.} [Apply $\ell$] If $q_i=\ell$, call $\matchlast(i)$. \\
\textbf{C42.} [Apply $\ell^{-1}$] If $q_i=\ell^{-1}$, call $\matchlast^{-1}(i)$. \\
\textbf{C43.} [Apply $f$] If $q_i=f$, then if $s_{m-i+1}\le m+i$, call $\matchfirst(i)$, otherwise call $\matchlast^{-1}(i+1)$, $\matchfirst(i)$, and $\matchlast(i+1)$. \\
\textbf{C44.} [Apply $f^{-1}$] If $q_i=f^{-1}$, then if $i=m$ or $c_{m-i}={*}$, call $\matchfirst^{-1}(i)$, otherwise call $\matchlast^{-1}(i+1)$, $\matchfirst^{-1}(i)$, and $\matchlast(i+1)$.
\item{} [Select next step]
If $q_i\in\{f,\ell\}$, set $l_i\leftarrow l_i-2$, otherwise set $l_i\leftarrow l_i+2$.
Also set $d_m\leftarrow m$. \\
If $b_i=-$, then go to~C6, else if $q_i\neq f^{-1}$, then go to~C7, otherwise go to~C8.
\item{} [Reached last descendant in case $|C|=0$]
Set $d_i\leftarrow d_{i-1}$ and $d_{i-1}\leftarrow i-1$.
If $o_i=+$, then set $l_i\leftarrow 2$, otherwise set $l_i\leftarrow 4$.
Also set $b_i\leftarrow +$, $o_i\leftarrow -o_i$, and $q_i\leftarrow f$, and go back to~C2.
\item{} [Descendants in case $|C|\geq 2$ remaining]
If $q_i\in\{\ell,\ell^{-1}\}$ set $q_i\leftarrow f^{-1}$, otherwise set $q_i\leftarrow \ell$ if $o_i=+$ and $q_i\leftarrow \ell^{-1}$ if $o_i=-$.
Go back to~C2.
\item{} [Reached last descendant in case $|C|\geq 2$]
Set $d_i\leftarrow d_{i-1}$, $d_{i-1}\leftarrow i-1$, and $j\leftarrow d_i$.
If $j>0$ and $q_j\in\{f,\ell\}$, set $l_i\leftarrow l_i-2$, otherwise set $l_i\leftarrow l_i+2$.
If $l_i=0$, set $b_i\leftarrow -$ and $q_i\leftarrow f^{-1}$, else if $l_i=2$ and $o_i=-$, set $b_i\leftarrow -$ and $q_i\leftarrow f$, otherwise set $q_i\leftarrow f$.
Also set $o_i\leftarrow -o_i$, and go back to~C2.
\end{enumerate}
\end{algo}

The algorithm uses the following four auxiliary functions, which implement the functions $f,f^{-1},\ell,\ell^{-1}$ on the chain~$C$; see Figure~\ref{fig:aux}.
The parameter $i\in\{1,\ldots,m\}$ determines the middle $2i$ entries of~$C$, namely $c_{m-i+1},\ldots,c_{m+i}$, that the function works on.
\begin{itemize}[leftmargin=4mm, noitemsep, topsep=3pt plus 3pt]
\item
$\matchfirst(i)$: Set $\alpha\leftarrow m-i+1$, $\beta\leftarrow s_\alpha$, and $\gamma\leftarrow s_\beta$.
Then set $c_\alpha\leftarrow 0$, $c_\beta\leftarrow 1$, $p_\alpha\leftarrow\beta$, $p_\beta\leftarrow \alpha$, $s_{\alpha-1}\leftarrow\gamma$, and $t_\gamma\leftarrow \alpha-1$.
\item
$\matchfirst^{-1}(i)$: Set $\alpha\leftarrow m-i+1$, $\beta\leftarrow p_\alpha$, and $\gamma\leftarrow s_{\alpha-1}$.
Then set $c_\alpha\leftarrow {*}$, $c_\beta\leftarrow {*}$, $s_{\alpha-1}\leftarrow \alpha$, $s_\alpha\leftarrow \beta$, $s_\beta\leftarrow \gamma$, $t_\gamma\leftarrow\beta$, $t_\beta\leftarrow\alpha$, and $t_\alpha\leftarrow\alpha-1$.
\item
$\matchlast(i)$: Set $\alpha\leftarrow m+i$, $\beta\leftarrow t_\alpha$, and $\gamma\leftarrow t_\beta$.
Then set $c_\beta\leftarrow 0$, $c_\alpha\leftarrow 1$, $p_\beta\leftarrow \alpha$, $p_\alpha\leftarrow \beta$, $s_\gamma\leftarrow\alpha+1$, and $t_{\alpha+1}\leftarrow \gamma$.
\item
$\matchlast^{-1}(i)$: Set $\alpha\leftarrow m+i$, $\beta\leftarrow p_{\alpha}$, and $\gamma\leftarrow t_{\alpha+1}$.
Then set $c_\beta\leftarrow {*}$, $c_\alpha\leftarrow {*}$, $s_\gamma\leftarrow \beta$, $s_\beta\leftarrow \alpha$, $s_\alpha\leftarrow \alpha+1$, $t_{\alpha+1}\leftarrow \alpha$, $t_\alpha\leftarrow \beta$, and $t_\beta\leftarrow \gamma$.
\end{itemize}

Note that the two subcases in steps~\textbf{C43} and~\textbf{C44} are captured by the relations~\eqref{eq:diff02} and~\eqref{eq:diff03} in Lemma~\ref{lem:conn-top0}, respectively.

\begin{figure}[h!]
\includegraphics{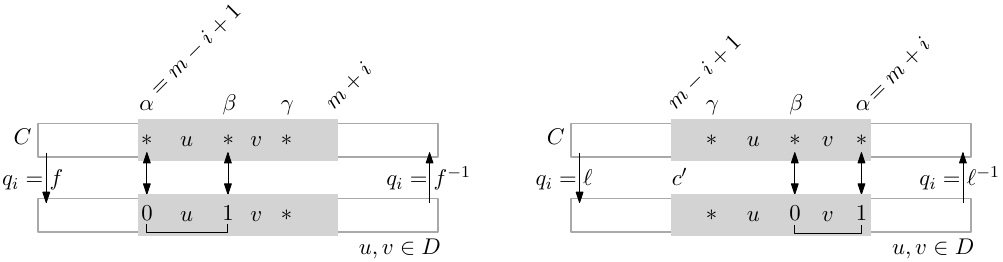}
\caption{Illustration of the four auxiliary functions used in Algorithm~C.
}
\label{fig:aux}
\end{figure}

\begin{table}[h!]
\caption{Protocol of Algorithm~C for~$n=2m=6$. Empty entries have unchanged values compared to the previous iteration/row.
Struck through values have changed twice in some iteration of the algorithm.
}
\makebox[0cm]{ 
\renewcommand*{\arraystretch}{0.9}
\begin{tabular}{l|c|c|c|c|c|c|c|c|c}
\hspace{6mm} & $C$ & $p$ & $s$ & $t$ & $l$ & $b$ & $o$ & $q$ & $d$ \\ \hline
1 & \bivc{$*$}{$*$}{$*$}{$*$}{$*$}{$*$} & \bivc{}{}{}{}{}{} & \bivd{1}{2}{3}{4}{5}{6}{7} & \bivd{0}{1}{2}{3}{4}{5}{6} & \bivav{2}{4}{6} & \biva{$-$}{$+$}{$+$} & \biva{$+$}{$+$}{$+$} & \bivaw{$f$}{$f$}{$f$} & \bivb{0}{1}{2}{3} \\
2 & \bivc{0}{1}{}{}{}{} & \bivc{2}{1}{}{}{}{} & \bivd{3}{}{}{}{}{}{} & \bivd{}{}{0}{}{}{}{} & \bivav{}{}{4} & \biva{}{}{} & \biva{}{}{} & \bivaw{}{}{$\ell$} & \bivb{}{}{}{} \\
3 & \bivc{}{}{}{}{0}{1} & \bivc{}{}{}{}{6}{5} & \bivd{}{}{}{}{7}{}{} & \bivd{}{}{}{}{}{}{4} & \bivav{}{}{2} & \biva{}{}{} & \biva{}{}{} & \bivaw{}{}{$f^{-1}$} & \bivb{}{}{}{} \\
4 & \bivc{$*$}{$*$}{}{}{}{} & \bivc{}{}{}{}{}{} & \bivd{1}{}{}{}{}{}{} & \bivd{}{}{2}{}{}{}{} & \bivav{}{}{\cancel{4}2} & \biva{}{}{} & \biva{}{}{$-$} & \bivaw{}{}{$f$} & \bivb{}{}{}{2} \\
5 & \bivc{}{0}{1}{}{}{} & \bivc{}{3}{2}{}{}{} & \bivd{}{4}{}{}{}{}{} & \bivd{}{}{}{1}{}{}{} & \bivav{}{2}{} & \biva{}{}{} & \biva{}{}{} & \bivaw{}{$\ell$}{} & \bivb{}{}{}{3} \\
6 & \bivc{0}{}{}{1}{}{} & \bivc{4}{}{}{1}{}{} & \bivd{7}{}{}{}{}{}{} & \bivd{}{}{}{}{}{}{0} & \bivav{}{}{0} & \biva{}{}{} & \biva{}{}{} & \bivaw{}{}{$\ell^{-1}$} & \bivb{}{}{}{} \\
7 & \bivc{}{}{}{}{$*$}{$*$} & \bivc{}{}{}{}{}{} & \bivd{5}{}{}{}{}{}{} & \bivd{}{}{}{}{0}{}{6} & \bivav{}{}{2} & \biva{}{}{} & \biva{}{}{} & \bivaw{}{}{$f^{-1}$} & \bivb{}{}{}{} \\
8 & \bivc{$*$}{}{}{$*$}{}{} & \bivc{}{}{}{}{}{} & \bivd{1}{}{}{}{5}{}{} & \bivd{}{}{}{}{4}{}{} & \bivav{}{}{\cancel{4}2} & \biva{}{}{$-$} & \biva{}{}{$+$} & \bivaw{}{}{$f$} & \bivb{}{}{}{2} \\
9 & \bivc{}{}{}{0}{1}{} & \bivc{}{}{}{5}{4}{} & \bivd{}{6}{}{}{}{}{} & \bivd{}{}{}{}{}{1}{} & \bivav{}{0}{} & \biva{}{}{} & \biva{}{}{} & \bivaw{}{$f^{-1}$}{} & \bivb{}{}{}{3} \\
10 & \bivc{0}{}{}{}{}{1} & \bivc{6}{}{}{}{}{1} & \bivd{7}{}{}{}{}{}{} & \bivd{}{}{}{}{}{}{0} & \bivav{}{}{\cancel{0}2} & \biva{}{}{$+$} & \biva{}{}{$-$} & \bivaw{}{}{} & \bivb{}{}{}{2} \\
11 & \bivc{$*$}{$*$}{0}{}{}{} & \bivc{}{}{6}{}{}{3} & \bivd{1}{2}{7}{}{}{}{} & \bivd{}{}{}{}{}{}{2} & \bivav{}{\cancel{2}0}{} & \biva{}{$-$}{} & \biva{}{$-$}{} & \bivaw{}{}{} & \bivb{}{}{1}{3} \\
12 & \bivc{0}{1}{}{}{}{} & \bivc{2}{1}{}{}{}{} & \bivd{7}{}{}{}{}{}{} & \bivd{}{}{}{}{}{}{0} & \bivav{}{}{0} & \biva{}{}{} & \biva{}{}{} & \bivaw{}{}{$\ell^{-1}$} & \bivb{}{}{}{} \\
13 & \bivc{}{}{$*$}{}{}{$*$} & \bivc{}{}{}{}{}{} & \bivd{3}{}{}{6}{}{}{} & \bivd{}{}{0}{}{}{3}{6} & \bivav{}{}{2} & \biva{}{}{} & \biva{}{}{} & \bivaw{}{}{$f^{-1}$} & \bivb{}{}{}{} \\
14 & \bivc{$*$}{$*$}{}{}{}{} & \bivc{}{}{}{}{}{} & \bivd{1}{}{3}{}{}{}{} & \bivd{}{}{2}{}{}{}{} & \bivav{}{}{\cancel{4}2} & \biva{}{}{$-$} & \biva{}{}{$+$} & \bivaw{}{}{$f$} & \bivb{}{}{2}{1} \\
15 & \bivc{}{0}{0}{1}{}{} & \bivc{}{5}{4}{3}{2}{} & \bivd{}{6}{}{}{}{}{} & \bivd{}{}{}{}{}{1}{} & \bivav{\cancel{0}2}{}{} & \biva{$+$}{}{} & \biva{$-$}{}{} & \bivaw{}{}{} & \bivb{}{0}{}{3} \\
16 & \bivc{0}{}{}{}{}{1} & \bivc{6}{}{}{}{}{1} & \bivd{7}{}{}{}{}{}{} & \bivd{}{}{}{}{}{}{0} & \bivav{}{}{\cancel{0}2} & \biva{}{}{$+$} & \biva{}{}{$-$} & \bivaw{}{}{} & \bivb{}{}{}{2} \\
17 & \bivc{$*$}{$*$}{}{}{0}{} & \bivc{}{}{}{}{6}{5} & \bivd{1}{2}{7}{}{}{}{} & \bivd{}{}{}{}{}{}{2} & \bivav{}{\cancel{2}4}{} & \biva{}{$+$}{} & \biva{}{$+$}{} & \bivaw{}{$f$}{} & \bivb{}{1}{0}{3} \\
18 & \bivc{0}{1}{}{}{}{} & \bivc{2}{1}{}{}{}{} & \bivd{7}{}{}{}{}{}{} & \bivd{}{}{}{}{}{}{0} & \bivav{}{}{0} & \biva{}{}{} & \biva{}{}{} & \bivaw{}{}{$\ell^{-1}$} & \bivb{}{}{}{} \\
19 & \bivc{}{}{}{}{$*$}{$*$} & \bivc{}{}{}{}{}{} & \bivd{5}{}{}{}{}{}{} & \bivd{}{}{}{}{0}{5}{6} & \bivav{}{}{2} & \biva{}{}{} & \biva{}{}{} & \bivaw{}{}{$f^{-1}$} & \bivb{}{}{}{} \\
20 & \bivc{$*$}{$*$}{}{}{}{} & \bivc{}{}{}{}{}{} & \bivd{1}{}{5}{}{}{}{} & \bivd{}{}{}{}{2}{}{} & \bivav{}{}{\cancel{4}6} & \biva{}{}{} & \biva{}{}{$+$} & \bivaw{}{}{$f$} & \bivb{}{}{2}{0} \\
\end{tabular}
}
\label{tab:algoC}
\end{table}

\section{Acknowledgements}

We thank Ji\v{r}\'{i} Fink for several valuable discussions about symmetric chain decompositions, and for feedback on an earlier draft of this paper.
We also thank the anonymous reviewers whose suggestions helped improving the presentation.

\bibliographystyle{alpha}
\bibliography{../refs}

\newcommand{\etalchar}[1]{$^{#1}$}
\begin{thebibliography}{dBvETK51}

\bibitem[Aig73]{MR0319772}
M.~Aigner.
\newblock Lexicographic matching in {B}oolean algebras.
\newblock {\em J. Combin. Theory Ser. B}, 14:187--194, 1973.

\bibitem[BER76]{MR0424386}
J.~Bitner, G.~Ehrlich, and E.~Reingold.
\newblock Efficient generation of the binary reflected {G}ray code and its
  applications.
\newblock {\em Comm. ACM}, 19(9):517--521, 1976.

\bibitem[Bol86]{MR866142}
B.~Bollob{\'a}s.
\newblock {\em Combinatorics: set systems, hypergraphs, families of vectors and
  combinatorial probability}.
\newblock Cambridge University Press, Cambridge, 1986.

\bibitem[BW84]{MR737262}
M.~Buck and D.~Wiedemann.
\newblock Gray codes with restricted density.
\newblock {\em Discrete Math.}, 48(2-3):163--171, 1984.

\bibitem[Che00]{MR1778200}
Y.~Chen.
\newblock Kneser graphs are {H}amiltonian for {$n\geq 3k$}.
\newblock {\em J. Combin. Theory Ser. B}, 80(1):69--79, 2000.

\bibitem[Che03]{MR1999733}
Y.~Chen.
\newblock Triangle-free {H}amiltonian {K}neser graphs.
\newblock {\em J. Combin. Theory Ser. B}, 89(1):1--16, 2003.

\bibitem[cos]{cos_chains}
The Combinatorial Object Server: \url{http://www.combos.org/chains}.

\bibitem[dBvETK51]{MR0043115}
N.~de~Bruijn, C.~van Ebbenhorst~Tengbergen, and D.~Kruyswijk.
\newblock On the set of divisors of a number.
\newblock {\em Nieuw Arch. Wiskunde (2)}, 23:191--193, 1951.

\bibitem[EG73]{MR382006}
P.~Erd{\H o}s and R.~K. Guy.
\newblock Crossing number problems.
\newblock {\em Amer. Math. Monthly}, 80:52--58, 1973.

\bibitem[EHH01]{MR1887372}
M.~El-Hashash and A.~Hassan.
\newblock On the {H}amiltonicity of two subgraphs of the hypercube.
\newblock In {\em Proceedings of the {T}hirty-second {S}outheastern
  {I}nternational {C}onference on {C}ombinatorics, {G}raph {T}heory and
  {C}omputing ({B}aton {R}ouge, {LA}, 2001)}, volume 148, pages 7--32, 2001.

\bibitem[Ehr73]{MR0366085}
G.~Ehrlich.
\newblock Loopless algorithms for generating permutations, combinations, and
  other combinatorial configurations.
\newblock {\em J. Assoc. Comput. Mach.}, 20:500--513, 1973.

\bibitem[FdFSV08]{MR2440735}
L.~Faria, C.~M.~H. de~Figueiredo, O.~S{\'y}kora, and I.~Vrt{'o}.
\newblock An improved upper bound on the crossing number of the hypercube.
\newblock {\em J. Graph Theory}, 59(2):145--161, 2008.

\bibitem[Fin07]{MR2354719}
J.~Fink.
\newblock Perfect matchings extend to {H}amilton cycles in hypercubes.
\newblock {\em J. Combin. Theory Ser. B}, 97(6):1074--1076, 2007.

\bibitem[Fin19]{MR3936192}
J.~Fink.
\newblock Matchings extend into 2-factors in hypercubes.
\newblock {\em Combinatorica}, 39(1):77--84, 2019.

\bibitem[FS13]{MR3043094}
T.~Feder and C.~Subi.
\newblock On hypercube labellings and antipodal monochromatic paths.
\newblock {\em Discrete Appl. Math.}, 161(10-11):1421--1426, 2013.

\bibitem[F{\"u}r85]{Furedi1985}
Z.~F{\"u}redi.
\newblock Problem session.
\newblock In {\em Kombinatorik geordneter Mengen}, Oberwolfach, BRD, 1985.

\bibitem[GJM{\etalchar{+}}21]{jaeger-et-al-journal:21}
P.~Gregor, S.~J{\"a}ger, T.~M{\"u}tze, J.~Sawada, and K.~Wille.
\newblock Gray codes and symmetric chains.
\newblock To appear in {\it J. Combin. Theory Ser. B}. Preprint available at
  {\it arXiv:1802.06021}, 2021.

\bibitem[GK76]{MR0389608}
C.~Greene and D.~J. Kleitman.
\newblock Strong versions of {S}perner's theorem.
\newblock {\em J. Combin. Theory Ser. A}, 20(1):80--88, 1976.

\bibitem[GKS04]{MR2034416}
J.~Griggs, C.~E. Killian, and C.~D. Savage.
\newblock Venn diagrams and symmetric chain decompositions in the {B}oolean
  lattice.
\newblock {\em Electron. J. Combin.}, 11(1):Paper 2, 30 pp., 2004.

\bibitem[GM18]{MR3758308}
P.~Gregor and T.~M{\"u}tze.
\newblock Trimming and gluing {G}ray codes.
\newblock {\em Theoret. Comput. Sci.}, 714:74--95, 2018.

\bibitem[GMM20]{DBLP:conf/icalp/GregorMM20}
P.~Gregor, O.~Mi{\v c}ka, and T.~M{\"u}tze.
\newblock On the central levels problem.
\newblock In A.~Czumaj, A.~Dawar, and E.~Merelli, editors, {\em 47th
  International Colloquium on Automata, Languages, and Programming, {ICALP}
  2020}, volume 168 of {\em LIPIcs}, pages 60:1--60:17. Schloss Dagstuhl, 2020.

\bibitem[GMN18]{gregor-muetze-nummenpalo:18}
P.~Gregor, T.~M{\"u}tze, and J.~Nummenpalo.
\newblock A short proof of the middle levels theorem.
\newblock {\em Discrete Analysis}, 2018:8:12 pp., 2018.

\bibitem[Gra53]{gray:patent}
F.~Gray.
\newblock Pulse code communication, 1953.
\newblock March 17, 1953 (filed Nov. 1947). U.S. Patent 2,632,058.

\bibitem[Gri93]{MR1183998}
J.~R. Griggs.
\newblock On the distribution of sums of residues.
\newblock {\em Bull. Amer. Math. Soc. (N.S.)}, 28(2):329--333, 1993.

\bibitem[G{\v S}10]{MR2609124}
P.~Gregor and R.~{\v S}krekovski.
\newblock On generalized middle-level problem.
\newblock {\em Inform. Sci.}, 180(12):2448--2457, 2010.

\bibitem[Hav83]{MR737021}
I.~Havel.
\newblock Semipaths in directed cubes.
\newblock In {\em Graphs and other combinatorial topics ({P}rague, 1982)},
  volume~59 of {\em Teubner-Texte Math.}, pages 101--108. Teubner, Leipzig,
  1983.

\bibitem[Hua19]{MR4024566}
H.~Huang.
\newblock Induced subgraphs of hypercubes and a proof of the sensitivity
  conjecture.
\newblock {\em Ann. of Math. (2)}, 190(3):949--955, 2019.

\bibitem[Hur94]{MR1271867}
G.~Hurlbert.
\newblock The antipodal layers problem.
\newblock {\em Discrete Math.}, 128(1-3):237--245, 1994.

\bibitem[Kle65]{MR184865}
D.~J. Kleitman.
\newblock On a lemma of {L}ittlewood and {O}fford on the distribution of
  certain sums.
\newblock {\em Math. Z.}, 90:251--259, 1965.

\bibitem[Knu11]{MR3444818}
D.~E. Knuth.
\newblock {\em The Art of Computer Programming. {V}ol. 4{A}. {C}ombinatorial
  Algorithms. {P}art 1}.
\newblock Addison-Wesley, Upper Saddle River, NJ, 2011.

\bibitem[KRSW04]{MR2114190}
C.~E. Killian, F.~Ruskey, C.~D. Savage, and M.~Weston.
\newblock Half-simple symmetric {V}enn diagrams.
\newblock {\em Electron. J. Combin.}, 11(1):Research Paper 86, 22, 2004.

\bibitem[KT88]{MR962224}
H.~A. Kierstead and W.~T. Trotter.
\newblock Explicit matchings in the middle levels of the {B}oolean lattice.
\newblock {\em Order}, 5(2):163--171, 1988.

\bibitem[LS03]{locke-stong:03}
S.~Locke and R.~Stong.
\newblock Problem 10892: Spanning cycles in hypercubes.
\newblock {\em Amer. Math. Monthly}, 110:440--441, 2003.

\bibitem[MN17]{MR3627876}
T.~M{\"u}tze and J.~Nummenpalo.
\newblock A constant-time algorithm for middle levels {G}ray codes.
\newblock In {\em Proceedings of the {T}wenty-{E}ighth {A}nnual {ACM}-{SIAM}
  {S}ymposium on {D}iscrete {A}lgorithms}, pages 2238--2253. SIAM,
  Philadelphia, PA, 2017.

\bibitem[MS17]{MR3759914}
T.~M{\"u}tze and P.~Su.
\newblock Bipartite {K}neser graphs are {H}amiltonian.
\newblock {\em Combinatorica}, 37(6):1207--1219, 2017.

\bibitem[M{\"u}t16]{MR3483129}
T.~M{\"u}tze.
\newblock Proof of the middle levels conjecture.
\newblock {\em Proc. Lond. Math. Soc.}, 112(4):677--713, 2016.

\bibitem[Nor08]{norine_2008}
S.~Norine.
\newblock Edge-antipodal colorings of cubes.
\newblock The Open Problem Garden. Available at {\it
  http://www.openproblemgarden.org/op/edge\_antipodal\_colorings\_of\_cubes},
  2008.

\bibitem[NS94]{MR1313531}
N.~Nisan and M.~Szegedy.
\newblock On the degree of {B}oolean functions as real polynomials.
\newblock {\em Comput. Complexity}, 4(4):301--313, 1994.
\newblock Special issue on circuit complexity (Barbados, 1992).

\bibitem[Pik99]{MR1765729}
O.~Pikhurko.
\newblock On edge decompositions of posets.
\newblock {\em Order}, 16(3):231--244 (2000), 1999.

\bibitem[RS93]{MR1201997}
F.~Ruskey and C.~D. Savage.
\newblock Hamilton cycles that extend transposition matchings in {C}ayley
  graphs of {$S_n$}.
\newblock {\em SIAM J. Discrete Math.}, 6(1):152--166, 1993.

\bibitem[RSW06]{MR2268388}
F.~Ruskey, C.~D. Savage, and S.~Wagon.
\newblock The search for simple symmetric {V}enn diagrams.
\newblock {\em Notices Amer. Math. Soc.}, 53(11):1304--1312, 2006.

\bibitem[Rus03]{ruskey_2003}
F.~Ruskey.
\newblock Combinatorial generation.
\newblock Book draft, 2003.

\bibitem[Sav93]{MR1275228}
C.~D. Savage.
\newblock Long cycles in the middle two levels of the {B}oolean lattice.
\newblock {\em Ars Combin.}, 35(A):97--108, 1993.

\bibitem[Sav97]{MR1491049}
C.~D. Savage.
\newblock A survey of combinatorial {G}ray codes.
\newblock {\em SIAM Rev.}, 39(4):605--629, 1997.

\bibitem[Sim91]{MR1152123}
J.~Simpson.
\newblock Hamiltonian bipartite graphs.
\newblock In {\em Proceedings of the {T}wenty-second {S}outheastern
  {C}onference on {C}ombinatorics, {G}raph {T}heory, and {C}omputing ({B}aton
  {R}ouge, {LA}, 1991)}, volume~85, pages 97--110, 1991.

\bibitem[SK79]{MR532807}
J.~Shearer and D.~J. Kleitman.
\newblock Probabilities of independent choices being ordered.
\newblock {\em Stud. Appl. Math.}, 60(3):271--276, 1979.

\bibitem[Spi19]{MR3952674}
H.~Spink.
\newblock Orthogonal symmetric chain decompositions of hypercubes.
\newblock {\em SIAM J. Discrete Math.}, 33(2):910--932, 2019.

\bibitem[ST14]{MR3268651}
N.~Streib and W.~T. Trotter.
\newblock Hamiltonian cycles and symmetric chains in {B}oolean lattices.
\newblock {\em Graphs Combin.}, 30(6):1565--1586, 2014.

\bibitem[Sta15]{MR3467982}
R.~P. Stanley.
\newblock {\em Catalan numbers}.
\newblock Cambridge University Press, New York, 2015.

\bibitem[SW95]{MR1329390}
C.~D. Savage and P.~Winkler.
\newblock Monotone {G}ray codes and the middle levels problem.
\newblock {\em J. Combin. Theory Ser. A}, 70(2):230--248, 1995.

\bibitem[SW19]{shen_williams_2019}
X.~S. Shen and A.~Williams.
\newblock A middle levels conjecture for multiset permutations with
  uniform-frequency.
\newblock Williams College Technical Report CSTR-201901. Available at {\it
  http://tmuetze.de/papers/sw1.pdf}, 2019.

\bibitem[TL73]{MR0349274}
D.~Tang and C.~Liu.
\newblock Distance-{$2$} cyclic chaining of constant-weight codes.
\newblock {\em IEEE Trans. Computers}, C-22:176--180, 1973.

\bibitem[Tom15]{MR3349520}
I.~Tomon.
\newblock On a conjecture of {F}{\"u}redi.
\newblock {\em European J. Combin.}, 49:1--12, 2015.

\bibitem[WW77]{MR0450151}
D.~E. White and S.~G. Williamson.
\newblock Recursive matching algorithms and linear orders on the subset
  lattice.
\newblock {\em J. Combin. Theory Ser. A}, 23(2):117--127, 1977.

\end{thebibliography}

\appendix
\section{Loopless algorithm for odd $n$}

\begin{algo}{C'}{Loopless Gray code for Greene-Kleitman chains in $Q_n$ for odd $n=2m+1$, $m\geq 0$}
The input of the algorithm is an integer~$m\geq 0$, and the produced chains in the cycle ordering~$\Lambda_{2m+1}$ are visited in step~C'2.
The algorithm uses the same data structures and auxiliary functions as Algorithm~C before.
The entries $b_i\in\{+,-\}$ now indicate whether we are in the first case of \eqref{eq:desc1} ($+$) or the second case ($-$).
Moreover, the entries of the array~$q$ may now also attain two other values in addition to $\{f,f^{-1},\ell,\ell^{-1}\}$, namely $g:=\ell^{-1}\circ f$ and $g^{-1}:=f^{-1}\circ \ell$, to account for the transitions between the last two chains in the second case of \eqref{eq:desc1}.
\begin{enumerate}[label={\bfseries C'\arabic*.}, leftmargin=11mm, noitemsep, topsep=3pt plus 3pt]
\item{} [Initialize] Set $c_i\leftarrow {*}$ for $1\le i\le n$, $s_i\leftarrow i+1$ for $0\le i\le n$, and $t_i\leftarrow i-1$ for $1\le i\le n+1$.
Also set $l_i\leftarrow 2i+1$ for $1\le i\le m$, $b_i\leftarrow +$ for $1\le i\le m$, $o_i\leftarrow +$ and $q_i\leftarrow \ell$ for $1\le i\le m$, $d_i\leftarrow i$ for $0\le i\le m$.
\item{} [Visit] Visit the chain $c_1\ldots c_n$.
\item{} [Select dimension] Set $i\leftarrow d_m$. Terminate if $i=0$.
\item{} [Perform operation] Depending on the value of $q_i$, branch to one of the following six cases. \\
\textbf{C'41.} [Apply $\ell$] If $q_i=\ell$, call $\matchlast(i)$. \\
\textbf{C'42.} [Apply $\ell^{-1}$] If $q_i=\ell^{-1}$, call $\matchlast^{-1}(i)$. \\
\textbf{C'43.} [Apply $f$] If $q_i=f$, call $\matchfirst(i)$. \\
\textbf{C'44.} [Apply $f^{-1}$] If $q_i=f^{-1}$, call $\matchfirst^{-1}(i)$. \\
\textbf{C'45.} [Apply $\ell^{-1}$, then $f$] If $q_i=g=\ell^{-1}\circ f$, then if $i=m$ call $\matchlast^{-1}(i)$ and $\matchfirst(i)$, otherwise call $\matchlast^{-1}(i)$, $\matchfirst(i)$, and $\matchfirst(i+1)$. \\
\textbf{C'46.} [Apply $f^{-1}$, then $\ell$] If $q_i=g^{-1}=f^{-1}\circ \ell$, then if $i=m$ call $\matchfirst^{-1}(i)$ and $\matchlast(i)$, otherwise call $\matchfirst^{-1}(i+1)$, $\matchfirst^{-1}(i)$, and $\matchlast(i)$.
\item{} [Select next step]
If $q_i\in\{f,\ell\}$, set $l_i\leftarrow l_i-2$, else if $q_i\in\{f^{-1},\ell^{-1}\}$, set $l_i\leftarrow l_i+2$.
Also set $d_m\leftarrow m$.
If $b_i=-$ and $q_i\in\{\ell,g^{-1}\}$, then go to~C6, else if $b_i=-$ and $q_i\in\{g,\ell^{-1}\}$, then go to~C7, else if $q_i\neq \ell^{-1}$, then go to~C8, otherwise go to~C9.
\item{} [Descendants in case $|C|=1$ remaining]
If $q_i=\ell$ set $q_i\leftarrow g$, otherwise set $q_i\leftarrow \ell^{-1}$.
Go back to~C2.
\item{} [Reached last descendant in case $|C|=1$]
Set $d_i\leftarrow d_{i-1}$, $d_{i-1}\leftarrow i-1$, and $j\leftarrow d_i$.
If $j>0$ and $q_j\in\{f,\ell,g\}$, set $l_i\leftarrow l_i-2$, otherwise set $l_i\leftarrow l_i+2$.
If $l_i=1$, set $q_i\leftarrow g^{-1}$, else if $l_i=3$ and $q_{i-1}=\ell^{-1}$, set $b_i\leftarrow +$, $o_i\leftarrow -$ and $q_i\leftarrow \ell$, else if $l_i=3$ and $q_{i-1}\neq \ell^{-1}$, set $q_i\leftarrow \ell$, otherwise set $b_i\leftarrow +$, $q_i\leftarrow \ell$ and $o_i\leftarrow -o_i$.
Go back to~C2.
\item{} [Descendants in case $|C|\geq 3$ remaining]
If $q_i\in\{f,f^{-1}\}$ set $q_i\leftarrow \ell^{-1}$, otherwise set $q_i\leftarrow f$ if $o_i=+$ and $q_i\leftarrow f^{-1}$ if $o_i=-$.
Go back to~C2.
\item{} [Reached last descendant in case $|C|\geq 3$]
Set $d_i\leftarrow d_{i-1}$, $d_{i-1}\leftarrow i-1$, and $j\leftarrow d_i$.
If $j>0$ and $q_j\in\{f,\ell,g\}$, set $l_i\leftarrow l_i-2$, otherwise set $l_i\leftarrow l_i+2$.
If $l_i=1$, set $b_i\leftarrow -$ and $q_i\leftarrow g^{-1}$, else if $l_i=3$ and $o_i=-$, set $b_i\leftarrow -$ and $q_i\leftarrow \ell$, otherwise set $q_i\leftarrow \ell$.
Also set $o_i\leftarrow -o_i$, and go back to~C2.
\end{enumerate}
\end{algo}

Note that the two subcases in steps~\textbf{C'45} and~\textbf{C'46} are captured by the left two equations and the right two equations in~\eqref{eq:diff13} in Lemma~\ref{lem:conn-1}, respectively.

\end{document}